\documentclass[a4paper,12pt]{article}

\usepackage{amsmath,amssymb,latexsym,amsfonts,amsthm}
\usepackage{mathrsfs}

\usepackage{graphicx}
\newcommand{\mail}{
\scalebox{0.6}{\includegraphics{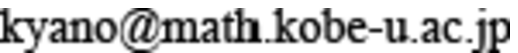}}}

\theoremstyle{plain}
\newtheorem{Thm}{Theorem}[section]

\newtheorem{Lem}[Thm]{Lemma}
\newtheorem{Prop}[Thm]{Proposition}
\newtheorem{Cor}[Thm]{Corollary}

\theoremstyle{definition}

\newtheorem{Rem}[Thm]{Remark}
\newtheorem{Ex}[Thm]{Example}

\newcommand{\bC}{\ensuremath{\mathbb{C}}}
\newcommand{\bD}{\ensuremath{\mathbb{D}}}

\newcommand{\bR}{\ensuremath{\mathbb{R}}}
\newcommand{\bS}{\ensuremath{\mathbb{S}}}
\newcommand{\bT}{\ensuremath{\mathbb{T}}}

\newcommand{\cB}{\ensuremath{\mathcal{B}}}
\newcommand{\cC}{\ensuremath{\mathcal{C}}}

\newcommand{\cF}{\ensuremath{\mathcal{F}}}
\newcommand{\cG}{\ensuremath{\mathcal{G}}}

\newcommand{\cM}{\ensuremath{\mathcal{M}}}

\newcommand{\cR}{\ensuremath{\mathcal{R}}}

\newcommand{\cT}{\ensuremath{\mathcal{T}}}
\newcommand{\cU}{\ensuremath{\mathcal{U}}}

\newcommand{\vm}{\ensuremath{\mbox{{\boldmath $m$}}}}
\newcommand{\vn}{\ensuremath{\mbox{{\boldmath $n$}}}}

\newcommand{\vp}{\ensuremath{\mbox{{\boldmath $p$}}}}

\renewcommand{\Re}{{\rm Re} \ }

\newcommand{\eps}{\ensuremath{\varepsilon}}
\newcommand{\e}{{\rm e}}
\renewcommand{\d}{{\rm d}}

\newcommand{\law}{\stackrel{{\rm law}}{=}}
\newcommand{\claw}{\stackrel{{\rm law}}{\longrightarrow}}

\renewcommand{\hat}{\widehat}
\renewcommand{\tilde}{\widetilde}

\newcommand{\abra}[1]{\left| #1 \right|}
\newcommand{\cbra}[1]{\left( #1 \right)}
\newcommand{\kbra}[1]{\left\{ #1 \right\}}
\newcommand{\ebra}[1]{\left[ #1 \right]}

\newcommand{\n}{\nonumber}

\numberwithin{equation}{section}

\newcounter{No}

\newcounter{Ci}[subsection]

\makeatletter
\renewcommand\section{\@startsection {section}{1}{\z@}%
                                   {-3.5ex \@plus -1ex \@minus -.2ex}%
                                   {2.3ex \@plus.2ex}%
                                   {\normalfont\large\bf}}
\makeatother

\newcommand{\fe}{\ensuremath{\mathfrak{e}}}

\newcounter{CO}

\newcounter{EO}

\newcommand{\ID}{{\rm (ID)}}
\newcommand{\SD}{{\rm (SD)}}
\newcommand{\GGC}{{\rm (GGC)}}

\makeatletter
\renewcommand\subsection{\@startsection {subsection}{1}{\z@}%
                                   {-3.5ex \@plus -1ex \@minus -.2ex}%
                                   {2.3ex \@plus.2ex}%
                                   {\normalfont\normalsize\bf}}
\makeatother

\begin{document}

%
\begin{center}
{\Large \bf 
On the laws of first hitting times of points 
for one-dimensional symmetric stable L\'evy processes 
}
\end{center}
\begin{center}
Kouji \textsc{Yano}\footnote{
Department of Mathematics, Graduate School of Science, 
Kobe University, Kobe, Japan. \mail}\footnote{
The research of this author is supported by KAKENHI (20740060)}, \qquad 
Yuko \textsc{Yano}\footnote{
Research Institute for Mathematical Sciences, Kyoto University, Kyoto, Japan.\label{foot: RIMS}} 
\qquad and \qquad 
Marc \textsc{Yor}\footnote{
Laboratoire de Probabilit\'es et Mod\`eles Al\'eatoires, Universit\'e Paris VI, Paris, France.}\footnote{
Institut Universitaire de France}\footnotemark[3]
\end{center}
\begin{center}
{\small \today}
\end{center}
\bigskip


\begin{abstract}
Several aspects of the laws of first hitting times of points 
are investigated 
for one-dimensional symmetric stable L\'evy processes. 
It\^o's excursion theory plays a key role in this study. 
\end{abstract}

\bigskip

\section{Introduction}

For one-dimensional Brownian motion, 
the laws of several random times, 
such as first hitting times of points and intervals, can be expressed explicitly 
in terms of elementary functions. 
Moreover, these laws are infinitely divisible (abbrev. as {\ID}), 
and in fact, self-decomposable (abbrev. as {\SD}). 

The aim of the present paper 
is to study various aspects of the laws of first hitting times of points and last exit times 
for one-dimensional symmetric stable L\'evy processes. 
We shall put some special emphasis on the following objects: 
(i) the laws of the ratio of two independent gamma variables, 
which, as is usual, we call {\em beta variables of the second kind}; 
(ii) harmonic transform of It\^o's measure of excursions away from the origin. 
The present study is motivated by 
a recent work \cite{YYY} by the authors 
about penalisations of symmetric stable L\'evy paths. 

The organisation of the present paper is as follows. 
In Section \ref{sec: alpha Cauchy}, 
we recall several facts concerning beta and gamma variables and their variants. 
In Section \ref{sec: exc}, 
we briefly recall It\^o's excursion theory 
and make some discussions about last exit times. 
In Section \ref{sec: harm}, 
we consider harmonic transforms of symmetric stable L\'evy processes, 
which plays an important role in our study. 
In Section \ref{sec: hitting}, 
we discuss the laws of first hitting times of single points 
and last exit times 
for symmetric stable L\'evy processes. 
In Section \ref{sec: hitting2}, 
we discuss the laws of those random times 
for the absolute value of symmetric stable L\'evy processes, 
which includes the reflecting Brownian motion as a special case.

\section{Preliminaries: several important random variables}
\label{sec: alpha Cauchy}

\subsection{Generalized gamma convolutions}

For $ a>0 $, we write $ \cG_a $ for a gamma variable with parameter $ a $: 
\begin{align}
P(\cG_a \in \d x) 
= \frac{1}{\Gamma(a)} x^{a-1} \e^{-x} \d x 
, \qquad x>0 . 
\label{}
\end{align}
As a rather general framework, 
we recall the class of {\em generalized gamma convolutions} (abbrev. as {\GGC}), 
which is an important subclass of {\SD}; namely, 
\begin{align}
\GGC \subset \SD \subset \ID . 
\label{}
\end{align}
A nice reference for details is the monograph \cite{MR1224674} by Bondesson. 
A recent survey can be found in James--Roynette--Yor \cite{JRY-GGC}.

A random variable $ X $ 
is said to be {\it of {\GGC} type} if 
it is a weak limit of linear combinations of independent gamma variables 
with positive coefficients.

\begin{Thm}[See, e.g., {\cite[Thm.3.1.1]{MR1224674}}]
A random variable $ X $ is of {\GGC} type if and only if 
there exist a non-negative constant $ a $ 
and a non-negative measure $ U(\d t) $ on $ (0,\infty ) $ 
with 
\begin{align}
\int_{(0,1]} |\log t| U(\d t) < \infty 
\qquad \text{and} \qquad 
\int_{(1,\infty )} \frac{1}{t} U(\d t) < \infty 
\label{}
\end{align}
such that 
\begin{align}
E[ \e^{-\lambda X} ] = 
\exp \kbra{ - a \lambda - \int_{(0,\infty )} \log (1+\lambda / t) U(\d t) } 
, \qquad \lambda > 0 . 
\label{}
\end{align}
\end{Thm}

In what follows 
we shall call $ U(\d t) $ the {\em Thorin measure} 
associated with $ X $. 

By simple calculations, it follows that 
\begin{align}
E[ \e^{-\lambda X} ] 
=& 
\exp \kbra{ - a \lambda + \int_0^{\infty } \cbra{ \frac{1}{s+\lambda} - \frac{1}{s} } 
U((0,s)) \d s } 
\\
=& 
\exp \kbra{ - a \lambda - \int_0^{\infty } \cbra{1-\e^{-\lambda u}} 
\frac{1}{u} \cbra{ \int_{(0,\infty )} \e^{-u t} U(\d t) } \d u } 
. 
\label{}
\end{align}
In particular, the following holds: 
{\it 
The law of $ X $ is of {\ID} type 
and its L\'evy measure has a density 
given by $ n(u) := \frac{1}{u} \int_{(0,\infty )} \e^{-u t} U(\d t) $. 
}
Since $ u n(u) $ is non-increasing, 
the law of $ X $ is of {\SD} type.

\begin{Thm}[see, e.g., {\cite[Thm.4.1.1 and 4.1.4]{MR1224674}}] \label{thm: GGC decomp}
Suppose that $ X $ is of {\GGC} type and that 
$ a=0 $ and $ b := U((0,\infty )) < \infty $. 
Then $ X $ may be represented as $ X \law \cG_b Y $ 
for some random variable $ Y $ independent of $ \cG_b $. 
The total mass of the Thorin measure is given by 
\begin{align}
b = \sup \kbra{ p \ge 0: \lim_{x \to 0+} \frac{\rho(x)}{x^{p-1}} = 0 } 
\label{eq: expression of Thorin meas}
\end{align}
where $ \rho $ is the density of the law of $ X $ with respect to the Lebesgue measure: 
\begin{align}
\rho(x) = \frac{1}{\Gamma(b)} x^{b-1} E \ebra{ \frac{1}{Y^b} \exp \cbra{ -\frac{x}{Y} } }. 
\label{}
\end{align}
\end{Thm}

\begin{Rem}
For a given $ X $, the law of the variable $ Y $ which represents $ X $ 
as in Theorem \ref{thm: GGC decomp} 
is unique; in fact, the gamma distribution is {\em simplifiable} 
(see \cite[Sec.1.12]{MR2016344}). 
\end{Rem}

\begin{Rem}
We do not know how to characterise explicitly 
the class of possible $ Y $'s which represent variables of {\GGC} type 
as in Theorem \ref{thm: GGC decomp}. 
As a partial converse, 
Bondesson (see \cite[Thm.6.2.1]{MR1224674}) has introduced a remarkable class 
which is closed under multiplication of independent gamma variables. 
\end{Rem}

\subsection{Beta and gamma variables}

We introduce notations 
and recall several basic facts concerning the beta and gamma variables. 
See \cite[Chap.4]{MR2016344} for details. 
For $ a,b>0 $, we write 
$ \cB_{a,b} $ for a beta variable (of the first kind) with parameters $ a,b $: 
\begin{align}
P(\cB_{a,b} \in \d x) 
= \frac{1}{B(a,b)} x^{a-1} (1-x)^{b-1} \d x 
, \qquad 0<x<1 
\label{}
\end{align}
where $ B(a,b) $ is the beta function: 
\begin{align}
B(a,b) = \frac{\Gamma(a) \Gamma(b)}{\Gamma(a+b)} . 
\label{}
\end{align}
Note that 
$ \cB_{a,b} \law 1-\cB_{b,a} $ for $ a,b>0 $ 
and that 
$ \cB_{a,1} \law \cU^{\frac{1}{a}} $ for $ a>0 $ 
where $ \cU $ is a uniform variable on $ (0,1) $. 
The following identity in law is well-known: 
{\it For any $ a,b>0 $, 
\begin{align}
\cbra{ \cG_a,\hat{\cG}_b } \law 
\cbra{ \cB_{a,b} \cG_{a+b} , (1-\cB_{a,b}) \cG_{a+b} } , 
\label{eq: Za Zb}
\end{align}
or equivalently, 
\begin{align}
\cbra{ \cG_a + \hat{\cG}_b, \frac{\cG_a}{\cG_a + \hat{\cG}_b} } \law 
\cbra{ \cG_{a+b} , \cB_{a,b} } 
\label{}
\end{align}
where on the left hand sides $ \cG_a $ and $ \hat{\cG}_b $ are independent 
and on the right hand sides $ \cB_{a,b} $ and $ \cG_{a+b} $ are independent. 
}
The proof is elementary; it can be seen in \cite[(4.2.1)]{MR2016344}, and so we omit it. 

Using the formula \eqref{eq: Za Zb}, we obtain another expression 
of the Thorin measure of a variable of {\GGC} type. 

\begin{Thm}
Under the same assumption as in Theorem \ref{thm: GGC decomp}, 
the total mass of the Thorin measure is given by 
\begin{align}
b = \inf \kbra{ c \ge 0 : X \law \cG_c Y_c 
\ \text{for some random variable $ Y_c $ independent of $ \cG_c $} } . 
\label{eq: another expression of Thorin measure}
\end{align}
\end{Thm}

\begin{proof}
Let us write $ \tilde{b} $ 
for the right hand side of \eqref{eq: another expression of Thorin measure}. 

By Theorem \ref{thm: GGC decomp}, we have $ X \law \cG_b Y $ 
for some random variable $ Y $ independent of $ \cG_b $. 
For any $ c>b $, we have $ \cG_b \law \cG_c \cB_{b,c-b} $ 
where $ \cG_c $ and $ \cB_{b,c-b} $ are independent, 
which implies that $ c \ge \tilde{b} $ for any such $ c $. 
Hence we obtain $ b \ge \tilde{b} $. 

Suppose that $ b > \tilde{b} $. 
Then we may take $ c $ with $ b>c>\tilde{b} $ such that 
$ X \law \cG_c Z $ for some random variable $ Z $ independent of $ \cG_c $. 
Then we have another expression of the density $ \rho(x) $ as 
\begin{align}
\rho(x) = \frac{1}{\Gamma(c)} x^{c-1} E \ebra{ \frac{1}{Z^c} \exp \cbra{ -\frac{x}{Z} } }. 
\label{}
\end{align}
By the monotone convergence theorem, this implies that 
\begin{align}
\lim_{x \to 0+} \frac{\rho(x)}{x^{c-1}} = \frac{1}{\Gamma(c)} E \ebra{ \frac{1}{Z^c} } > 0 , 
\label{}
\end{align}
which shows that $ c \ge b $ by the formula \eqref{eq: expression of Thorin meas}. 
This leads to a contradiction. 
Therefore we conclude that $ b=\tilde{b} $. 
\end{proof}

\subsection{Beta variables of the second kind}

Let us consider the ratios of two independent gamma variables, 
which are sometimes called {\em beta variables of the second kind} 
or {\em beta prime variables}. 
By the identity \eqref{eq: Za Zb}, the following is obvious: 
{\it For any $ a,b>0 $,} 
\begin{align}
\frac{\cG_a}{\hat{\cG}_b} \law \frac{\cB_{a,b}}{1-\cB_{a,b}} . 
\label{}
\end{align}
The law of the ratio $ \cG_a/\hat{\cG}_b $ is given as follows: 
{\it For any $ a,b>0 $,} 
\begin{align}
P \cbra{ \frac{\cG_a}{\hat{\cG}_b} \in \d x } 
= \frac{1}{B(a,b)} \frac{x^{a-1}}{(1+x)^{a+b}} \d x 
, \qquad x>0 . 
\label{}
\end{align}

In spite of its simple statement, 
the following theorem is rather difficult to prove. 

\begin{Thm}[see, e.g., {\cite[Ex.4.3.1]{MR1224674}}]
For any $ a,b>0 $, 
the ratio $ \cG_a/\hat{\cG}_b $ is of {\GGC} type. 
Its Thorin measure has total mass $ a $. 
\end{Thm}

For the proof, see \cite{MR1224674}. 
We omit the details.

\subsection{$ \alpha $-Cauchy variables and Linnik variables}

It is well-known that 
the standard Cauchy distribution $ \frac{1}{\pi} \frac{1}{1+x^2} \d x $ 
and 
the bilateral exponential distribution $ \frac{1}{2} \e^{-|x|} \d x $ 
satisfy the following relation: 
\begin{quote}
{\bf (R)} The characteristic function of one of the two distributions 
is proportional to the density of the other. 
\end{quote}
We shall introduce $ \alpha $-analogues of these two distributions 
which satisfy the relation {\bf (R)}. 

Let us introduce the $ \alpha $-analogue for $ \alpha >1 $ 
of the standard Cauchy variable $ \cC $, which, as just recalled, is given by 
\begin{align}
P( \cC \in \d x ) = \frac{1}{\pi} \frac{1}{1+x^2} \d x 
, \qquad x \in \bR . 
\label{}
\end{align}
We define the {\em $ \alpha $-Cauchy variable} $ \cC_{\alpha } $ 
as follows: 
\begin{align}
P(\cC_{\alpha } \in \d x) 
= \frac{\sin (\pi/\alpha )}{2 \pi/\alpha } \frac{1}{1+|x|^{\alpha }} \d x 
, \qquad x \in \bR . 
\label{}
\end{align}
Note that $ \cC_2 \law \cC $. 
By a change of variables, the following is easy to see: 
{\it 
For $ \alpha >1 $, let $ \gamma = 1/\alpha \in (0,1) $. 
Then it holds that 
\begin{align}
\cC_{\alpha } 
= \epsilon \cbra{ \frac{\cG_{\gamma}}{\hat{\cG}_{1-\gamma}} }^{\gamma} 
\label{}
\end{align}
where $ \epsilon $ is a Bernoulli variable: $ P(\epsilon=1) = P(\epsilon=-1) = 1/2 $ 
independent of $ \cG_{\gamma} $ and $ \hat{\cG}_{1-\gamma} $. 
In particular, 
\begin{align}
\frac{\cG_{\gamma}}{\hat{\cG}_{1-\gamma}} 
= |\cC_{\alpha }|^{\alpha } . 
\label{}
\end{align}
}

Note that the law of a standard Cauchy variable $ \cC_2 $ is of {\SD} type. 
Moreover, the following theorem is known: 

\begin{Thm}[Bondesson \cite{MR943583}]
For $ 1<\alpha \le 2 $, 
the law of $ |\cC_{\alpha }| $ is of {\ID} type. 
\end{Thm}

It is easy to see that 
\begin{align}
|\cC_{\alpha }| \claw \cU 
\qquad \text{as} \ \alpha \to \infty 
\label{}
\end{align}
where $ \cU $ is a uniform variable on $ (0,1) $. 

\begin{Thm}[Thorin \cite{MR0431333}]
For $ p>0 $, 
the law of $ \cU^{-p} $, 
which is called the {\em Pareto distribution of index $ p $}, 
is of {\GGC} type. 
\end{Thm}

\begin{Rem}
The following problems still remain open: 
\\ \quad 
{\rm (i)} 
{\it Is it true that the law of $ \cC_{\alpha } $ 
is of {\SD} type (or of {\ID} type at least)?} 
\\ \quad 
{\rm (ii)} 
{\it Is it true that the law of $ |\cC_{\alpha }| $ is of {\SD} type?}
\\ \quad 
{\rm (iii)} 
{\it Is it true that the law of $ |\cC_{\alpha }|^{-p} $ for $ p>0 $ 
is of {\SD} type (or of {\ID} type at least)?}
\end{Rem}

\begin{Rem}
Bourgade--Fujita--Yor (\cite{MR2300217}) 
have proposed a new probabilistic method 
of computing special values of the Riemann zeta function $ \zeta(2n) $ 
via the Cauchy variable. 
Fujita--Y.~Yano--Yor \cite{FYY} 
have recently generalized their method via the $ \alpha $-Cauchy variables 
and obtained a probabilistic method 
for computing special values of the complementary sum of the Hurwitz zeta function: 
$ \zeta(2n,\gamma) + \zeta(2n,1-\gamma) $ for $ 0<\gamma<1 $. 
\end{Rem}

Following \cite{MR1047827}, 
we introduce the {\em Linnik variable} $ \Lambda_{\alpha } $ 
of index $ 0<\alpha \le 2 $ 
as follows: 
\begin{align}
E[ \e^{i \theta \Lambda_{\alpha } } ] = \frac{1}{1+|\theta|^{\alpha }} 
, \qquad \theta \in \bR . 
\label{eq: Linnik1}
\end{align}
It is easy to see that 
\begin{align}
\Lambda_{\alpha } \law X_{\alpha }(\fe) 
\label{}
\end{align}
where 
$ X_{\alpha }=(X_{\alpha }(t): t \ge 0) $ is 
the symmetric stable L\'evy process of index $ \alpha $ starting from 0 
such that 
\begin{align}
P[\e^{i \theta X_{\alpha }(t) }] = \e^{-t|\theta|^{\alpha }} 
, \qquad \theta \in \bR 
\label{eq: char func of X alpha}
\end{align}
and $ \fe $ is a standard exponential variable independent of $ X_{\alpha } $. 
Hence the laws of Linnik variables are of {\SD} type. 
A L\'evy process $ (\Lambda_{\alpha }(t)) $ 
with $ \Lambda_{\alpha }(1) \law \Lambda_{\alpha } $ 
is called a {\em Linnik process}; its characteristic function is: 
\begin{align}
E \ebra{ \e^{i \theta \Lambda_{\alpha }(t)} } 
= \frac{1}{(1+|\theta|^{\alpha })^t} 
, \qquad \theta \in \bR . 
\label{}
\end{align}
See James \cite{J2} for his study of Linnik processes. 
Note that the law of $ \Lambda_{\alpha } $ has a continuous density $ L_{\alpha }(x) $, 
i.e., 
\begin{align}
P(\Lambda_{\alpha } \in \d x) = L_{\alpha }(x) \d x . 
\label{eq: Linnik2}
\end{align}

\begin{Prop}
Suppose that $ 1<\alpha <2 $. Then 
the $ \alpha $-Cauchy distribution and the Linnik distribution of index $ \alpha $ 
satisfy the relation {\bf (R)}. 
\end{Prop}

\begin{proof}
Note that the identities \eqref{eq: Linnik1} and \eqref{eq: Linnik2} show that 
\begin{align}
\int_{-\infty }^{\infty } \e^{i \theta x} L_{\alpha }(x) \d x 
= \frac{1}{1+|\theta|^{\alpha }} 
, \qquad \theta \in \bR . 
\label{}
\end{align}
By Fourier inversion, we obtain: 
\begin{align}
L_{\alpha }(x) 
= \frac{1}{2 \pi} \int_{-\infty }^{\infty } \e^{- i x \theta } 
\frac{1}{1+|\theta|^{\alpha }} \d \theta 
, \qquad x \in \bR . 
\label{}
\end{align}
Hence: 
\begin{align}
E[ \e^{i \theta \cC_{\alpha }} ] 
= \frac{\sin (\pi/\alpha )}{2 \pi / \alpha } L_{\alpha }(\theta) 
, \qquad \theta \in \bR . 
\label{}
\end{align}
Now the proof is complete. 
\end{proof}

\subsection{Log-gamma processes and their variants}

We recall the classes of log-gamma processes, $ z $-processes and Meixner processes. 

It is well-known (see, e.g., \cite{MR0436267}) 
that the law of the logarithm of a gamma variable $ \log \cG_a $ 
is of {\SD} type. 
Let us introduce a L\'evy process $ (\eta_a(t):t \ge 0) $ such that 
\begin{align}
\eta_a(1) \law \log \cG_a . 
\label{eq: log cGa}
\end{align}
Following Carmona--Petit--Yor \cite{MR1648657}, 
we call the process $ (\eta_a(t):t \ge 0) $ the {\em log-gamma process}. 
Please be careful not to confuse with the convention that 
log-normal variables stand for exponentials of normal variables. 
In \eqref{eq: log cGa}, we simply take the logarithm of a gamma variable. 

The L\'evy characteristics of $ (\eta_a(t):t \ge 0) $ are given as follows. 

\begin{Thm}[see {\cite{MR1648657}} 
and also {\cite{MR2019545}}] 
For any $ a>0 $, the log-gamma process is represented as 
\begin{align}
\eta_a(t) \law 
t \Gamma'(1) + \sum_{j=0}^{\infty } \kbra{ \frac{t}{j+1} - \frac{\gamma^{(j)}(t)}{j+a} } 
\label{eq: loggamma}
\end{align}
where $ \gamma^{(0)},\gamma^{(1)},\ldots $ are independent gamma processes. 
In particular, the L\'evy exponent of $ (\eta_a(t):t \ge 0) $ defined by 
\begin{align}
E \ebra{ \e^{ i \theta \eta_a(t) } } 
= \cbra{ \frac{\Gamma(a+i\theta)}{\Gamma(a)} }^t 
= \e^{ t \phi_a(\theta) } 
\label{}
\end{align}
admits the representation 
\begin{align}
\phi_a(\theta) 
=& \log \frac{\Gamma(a+i\theta)}{\Gamma(a)} 
\\
=& i \theta \psi(a) + \int_{-\infty }^0 \cbra{\e^{i \theta u} - 1 - i \theta u } 
\frac{\e^{-a|u|}}{|u|(1-\e^{-|u|})} \d u 
\label{}
\end{align}
where $ \psi(z) = \Gamma'(z)/\Gamma(z) $ is called the digamma function. 
\end{Thm}

Let $ (\eta_a(t):t \ge 0) $ and $ (\hat{\eta}_b(t):t \ge 0) $ 
be independent log-gamma processes. 
Then the difference $ (\eta_a(t)-\hat{\eta}_b(t):t \ge 0) $ 
is called a {\em generalized $ z $-process} 
(see \cite{MR1874897}). 
In particular, we have 
\begin{align}
\eta_a(1)-\hat{\eta}_b(1) 
\law 
\log \frac{\cG_a}{\hat{\cG}_b} 
\label{}
\end{align}
and this law is called a {$ z $-distribution}. 
Its characteristic function is given by 
\begin{align}
E \ebra{ \exp \kbra{ i \theta \log \frac{\cG_a}{\hat{\cG}_b} } } 
= \frac{B(a+i\theta,b-i\theta)}{B(a,b)} 
\label{}
\end{align}
and the law itself is given by 
\begin{align}
P \cbra{ \log \frac{\cG_a}{\hat{\cG}_b} \in \d x } 
= \frac{1}{B(a,b)} \frac{\e^{ax}}{(1+\e^x)^{a+b}} \d x . 
\label{}
\end{align}

\subsection{Symmetric $ z $-processes}

We now consider a particular case of symmetric $ z $-processes, i.e., 
\begin{align}
\sigma_a(t) = \frac{1}{\pi} \kbra{ \eta_a(t) - \hat{\eta}_a(t) } 
, \qquad t \ge 0 . 
\label{}
\end{align}

We introduce a subordinator given by 
\begin{align}
\Sigma_a(t) 
= \frac{2}{\pi^2} \sum_{j=0}^{\infty } \frac{\gamma^{(j)}(t)}{(j+a)^2} 
\label{}
\end{align}
where $ \gamma^{(0)},\gamma^{(1)},\ldots $ are independent gamma processes. 
The L\'evy measure of $ (\Sigma_a(t):t \ge 0) $ 
may be obtained from the following: 
\begin{align}
E \ebra{ \e^{ -\lambda \Sigma_a(t) } } 
=& \prod_{j=0}^{\infty } E \ebra{ \exp \kbra{ -\lambda \frac{2 \gamma^{(j)}(t)}{\pi^2 (j+a)^2} } } 
\\
=& \prod_{j=0}^{\infty } \exp \kbra{ - t \int_0^{\infty } 
\cbra{ 1-\exp \cbra{ - \frac{2 \lambda u}{\pi^2 (j+a)^2} } } \e^{-u} \frac{\d u}{u} } 
\\
=& \prod_{j=0}^{\infty } \exp \kbra{ - t \int_0^{\infty } 
\cbra{ 1-\e^{- \lambda u} } \exp \cbra{- \frac{\pi^2 (j+a)^2}{2} u } \frac{\d u}{u} } 
\\
=& \exp \kbra{ - t \int_0^{\infty } 
\cbra{ 1-\e^{- \lambda u} } n(u) \d u } 
\label{}
\end{align}
where 
\begin{align}
n(u) = \frac{1}{u} \sum_{j=0}^{\infty } 
\exp \cbra{- \frac{\pi^2 (j+a)^2}{2} u } 
= \frac{1}{u} \int_{(0,\infty )} \e^{-ut} U(\d t) 
\label{}
\end{align}
with 
\begin{align}
U = \sum_{j=0}^{\infty } \delta_{\pi^2 (j+a)^2/2} . 
\label{}
\end{align}
Hence we conclude that the law of $ \Sigma_a(t) $ for fixed $ t $ is of {\GGC} type.

The following theorem, due to Barndorff-Nielsen--Kent--S{\o}rensen \cite{MR678296}, 
connects the two L\'evy processes $ \Sigma_a $ and $ \sigma_a $: 

\begin{Thm}[{\cite{MR678296}}; 
see, e.g., {\cite{MR1648657}}] 
The process $ (\sigma_a(t):t \ge 0) $ 
may be obtained as the subordination of a Brownian motion $ (\hat{B}(u):u \ge 0) $ 
with respect to the subordinator $ (\Sigma_a(t):t \ge 0) $: 
\begin{align}
\sigma_a(t) \law \hat{B}(\Sigma_a(t)) 
, \qquad t \ge 0 . 
\label{}
\end{align}
\end{Thm}

\begin{proof}
By \eqref{eq: loggamma}, 
the process $ (\sigma_a(t):t \ge 0) $ is represented as 
\begin{align}
\sigma_a(t) 
= \frac{1}{\pi} 
\sum_{j=0}^{\infty } \frac{\hat{\gamma}^{(j)}(t) - \gamma^{(j)}(t)}{j+a} 
\label{}
\end{align}
where $ \gamma^{(0)},\gamma^{(1)},\ldots $, $ \hat{\gamma}^{(0)},\hat{\gamma}^{(1)},\ldots $ 
are independent gamma processes. 
Note that 
\begin{align}
\hat{\gamma}^{(j)}(t) - \gamma^{(j)}(t) 
\law \Lambda_2(t) 
\law \sqrt{2} \hat{B}(\gamma(t)) 
\label{}
\end{align}
where $ (\Lambda_2(t):t \ge 0) $ is a Linnik process, 
i.e., a L\'evy process such that $ \Lambda_2(1) \law \Lambda_2 $. 
Now we obtain 
\begin{align}
\hat{B}(\Sigma_a(t)) 
\law& \frac{\sqrt{2}}{\pi} \sum_{j=0}^{\infty } \frac{\hat{B}_j(\gamma^{(j)}(t))}{j+a} 
\label{eq: proof of hatB Sigmaat} \\
\law& \frac{1}{\pi} \sum_{j=0}^{\infty } \frac{\hat{\gamma}^{(j)}(t) - \gamma^{(j)}(t) }{j+a} 
\\
\law& \sigma_a(t) 
\label{}
\end{align}
where on the right hand side of \eqref{eq: proof of hatB Sigmaat} 
the $ \hat{B}_j $'s are independent of $ \gamma^{(j)} $'s. 
The proof is completed. 
\end{proof}

The characteristic function of $ \sigma_a(t) $ is given by 
\begin{align}
E[ \e^{i \theta \sigma_a(t)} ] 
= \cbra{ E \ebra{ \exp \kbra{ i \frac{\theta}{\pi} \log \frac{\cG_a}{\hat{\cG}_a} } } }^t 
= \e^{ - t \Phi_a(\theta) } 
\label{}
\end{align}
where 
\begin{align}
\Phi_a(\theta) 
=& \int_{-\infty }^{\infty } \cbra{ 1-\e^{i \theta u} } 
\frac{\e^{-a \pi |u|}}{|u|(1-\e^{- \pi |u|})} \d u 
\\
=& 2 \int_0^{\infty } \cbra{ 1-\cos \theta u } 
\frac{\e^{-a \pi u}}{u(1-\e^{- \pi u})} \d u . 
\label{eq: definition of Phi}
\end{align}
For $ t=1 $, the law of $ \sigma_a(1) $ is given by 
\begin{align}
P \cbra{ \sigma_a(1) \in \d x } 
= P \cbra{ \frac{1}{\pi} \log \frac{\cG_a}{\hat{\cG}_a} \in \d x } 
= \frac{\pi}{B(a,a)} \frac{\e^{a \pi x}}{(1+\e^{\pi x})^{2a}} \d x . 
\label{}
\end{align}

\begin{Ex}
When $ a=1/2 $, let us denote 
\begin{align}
C_t:= \Sigma_{\frac{1}{2}}(t) 
, \qquad 
\bC_t := \hat{B}(C_t) = \sigma_{\frac{1}{2}}(t) . 
\label{}
\end{align}
The law of $ \bC_1 $ 
is called the {\em hyperbolic cosine distribution}: 
\begin{align}
E[ \e^{i \theta \bC_1} ] 
= E[ \e^{- \frac{1}{2} \theta^2 C_1} ] 
= \frac{1}{\cosh \theta} 
, \qquad 
P \cbra{ \bC_1 \in \d x } 
= \frac{1}{\cosh \pi x} \d x . 
\label{}
\end{align}
Consequently, $ \bC_1 $ and $ \pi \bC_1 $ satisfy the relation {\bf (R)}. 
\end{Ex}

\begin{Ex}
When $ a=1 $, let us denote 
\begin{align}
S_t:= \Sigma_1(t) 
, \qquad 
\bS_t := \hat{B}(S_t) = \sigma_1(t) . 
\label{}
\end{align}
The law of $ \bS_1 $ 
is called the {\em logistic distribution}: 
\begin{align}
E[ \e^{i \theta \bS_1} ] 
= E[ \e^{- \frac{1}{2} \theta^2 S_1} ] 
= \frac{\theta}{\sinh \theta} 
, \qquad 
P \cbra{ \bS_1 \in \d x } 
= \frac{\pi}{(\cosh \pi x)^2} \d x . 
\label{}
\end{align}
Consequently, $ \bS_1 $ and $ \pi \bC_2 $ satisfy the relation {\bf (R)}. 
\end{Ex}

Let us introduce a subordinator $ (T_t) $ and then $ (\bT_t) $ such that 
\begin{align}
\bT_t = \hat{B}(T_t) 
\label{}
\end{align}
and that 
\begin{align}
E[ \e^{i \theta \bT_t} ] 
= E[ \e^{- \frac{1}{2} \theta^2 T_t} ] 
= \cbra{ \frac{\tanh \theta}{\theta} }^t . 
\label{}
\end{align}
It is well-known that the law of $ T_1 $ is of {\ID} type 
and hence that such processes exist. 
Now it is obvious that 
\begin{align}
C_t \law T_t + S_t 
, \qquad 
\bC_t \law \bT_t + \bS_t 
\label{}
\end{align}
where $ (T_t) $ and $ (S_t) $ are independent 
and so are $ (\bT_t) $ and $ (\bS_t) $. 

For further study of these processes $ C_t $, $ S_t $ and $ T_t $, 
see Pitman--Yor \cite{MR1969794}. 
By taking Laplace inversion, 
the density of the law of $ T_1 $ can be obtained in terms of theta function; 
see Knight \cite[Cor.2.1]{MR0253424} for details.

\subsection{Meixner processes}

Let $ \beta \in (-\pi,\pi) $ and let $ (\cM_{\beta}(t):t \ge 0) $ be a L\'evy process 
such that 
\begin{align}
\cM_{\beta}(1) \law \frac{1}{2 \pi} \log \frac{\cG_a}{\hat{\cG}_{1-a}} 
\qquad \text{where} \ 
\beta = (2a-1) \pi . 
\label{}
\end{align}
The law of $ \cM_{\beta }(t) $ for fixed $ t $ 
is called a {\em Meixner distribution} 
because of its close relation to {\em Meixner--Pollaczek polynomials} 
(See \cite{MR1761401}, 
\cite{Sch}, 
\cite{MR1617536} 
and 
\cite{MR1711971}). 
The characteristic function of $ \cM_{\beta}(t) $ is given by 
\begin{align}
E \ebra{ \e^{i \theta \cM_{\beta }(t)} } 
= \cbra{ \frac{\cos \frac{\beta }{2}}{\cosh \frac{\theta - i \beta }{2}} }^t 
= \e^{ t \xi_{\beta}(\theta) } 
\label{}
\end{align}
where 
\begin{align}
\xi_{\beta}(\theta) = \frac{i \theta}{2 \pi} 
\kbra{ \psi \cbra{ \frac{\pi + \beta}{2\pi} }-\psi \cbra{ \frac{\pi - \beta}{2 \pi} } } 
+ \int_{-\infty }^{\infty } \cbra{ \e^{i \theta u} - 1 - i \theta u } 
\frac{\e^{\beta u}}{u \sinh (\pi u)} \d u . 
\label{}
\end{align}
The law of $ \cM_{\beta}(t) $ itself is given by 
\begin{align}
P(\cM_{\beta }(t) \in \d x) 
=& \frac{\cbra{2 \cos \frac{\beta }{2}}^t}{2 \pi \Gamma(t)} 
\e^{\beta x} 
\abra{ \Gamma \cbra{ \frac{t}{2} + ix } }^2 \d x 
\\
=& \frac{\cbra{2 \cos \frac{\beta }{2}}^t}{2 \pi \Gamma(t)} 
\e^{\beta x} 
\Gamma(t/2)^2 \e^{ -\Phi_{t/2}(x) } \d x 
\\
=& \frac{\cbra{2 \cos \frac{\beta }{2}}^t B(t/2,t/2)}{2 \pi } 
\e^{ \beta x - \Phi_{t/2}(x) } 
\d x 
\label{}
\end{align}
where $ \Phi_a(x) $ has been defined in \eqref{eq: definition of Phi}. 

We simply write $ \cM_{\beta} $ for $ \cM_{\beta}(1) $. 
Remark that this Meixner distribution is identical 
to that of the log of an $ \alpha $-Cauchy variable: 
\begin{align}
\cM_{\beta} \law \frac{\alpha }{2 \pi} \log |\cC_{\alpha }| 
\qquad \text{where} \ 
\beta = \cbra{ \frac{2}{\alpha }-1 } \pi . 
\label{}
\end{align}
Remark also that the law of $ \cM_{\beta} $ is symmetric 
if and only if $ \beta=0 $ (or $ a = 1/2 $). 
Then the corresponding Meixner distribution 
is identical to 
the hyperbolic cosine distribution, 
up to the factor $ 1/2 $; precisely: 
\begin{align}
\cM_0 
\law \frac{1}{\pi} \log |\cC| 
\law \frac{1}{2 \pi} \log \frac{\cG_{1/2}}{\hat{\cG}_{1/2}} 
\law \frac{1}{2} \bC_1 . 
\label{}
\end{align}

\subsection{$ \alpha $-Rayleigh distributions}

For an exponential variable $ \fe $, the random variable 
\begin{align}
\cR = \sqrt{2 \fe} 
\label{}
\end{align}
is sometimes called a {\em Rayleigh variable}. 
We shall introduce an $ \alpha $-analogue of the Rayleigh variable. 

Let $ 0<\alpha \le 2 $ and let $ t>0 $. 
By Fourier inversion, 
we obtain from \eqref{eq: char func of X alpha} that 
\begin{align}
P(X_{\alpha }(t) \in \d x) = p^{(\alpha )}_t(x) \d x 
\label{}
\end{align}
where 
\begin{align}
p^{(\alpha )}_t(x) 
= \frac{1}{2\pi} \int_{-\infty }^{\infty } 
\e^{-i x \xi } \e^{-t|\xi|^{\alpha }} \d \xi 
= \frac{1}{\pi} \int_0^{\infty } 
\cos( x \xi ) \e^{-t \xi^{\alpha }} \d \xi . 
\label{}
\end{align}
Note that 
\begin{align}
p^{(\alpha )}_1(0) = \frac{\Gamma(1/\alpha )}{\alpha \pi} . 
\label{eq: p alpha 1 0}
\end{align}

\begin{Lem} \label{lem: K alpha}
Let $ 0 < \alpha \le 2 $. 
Then there exists a non-negative random variable $ \cR_{\alpha } $ such that 
\begin{align}
P( \cR_{\alpha } > x ) = \frac{p^{(\alpha )}_1(x)}{p^{(\alpha )}_1(0)} 
, \qquad x>0 . 
\label{eq: def of cR alpha}
\end{align}
In particular, $ \cR_2 = \sqrt{2} \cR = 2 \sqrt{\fe} $. 
\end{Lem}

We call $ \cR_{\alpha } $ an {\em $ \alpha $-Rayleigh variable} 
and its law the {\em $ \alpha $-Rayleigh distribution}. 

For the proof of Lemma \ref{lem: K alpha}, we introduce some notations. 
For $ 0<\alpha <1 $, we denote by $ \cT_{\alpha} $ 
the {\em unilateral $ \alpha $-stable distribution}: 
\begin{align}
E \ebra{ \e^{- \lambda \cT_{\alpha}} } = \e^{- \lambda^{\alpha }} 
, \qquad \lambda \ge 0 . 
\label{}
\end{align}
We denote by $ \cT'_{\alpha} $ the $ h $-size biased variable of $ \cT_{\alpha } $
with respect to $ h(x) = x^{-1/2} $: 
\begin{align}
E \ebra{ f(\cT'_{\alpha}) } 
= 
\frac{E \ebra{ (\cT_{\alpha} )^{-1/2} f(\cT_{\alpha}) } }
{E \ebra{ (\cT_{\alpha} )^{-1/2} } } 
\label{}
\end{align}
for any non-negative Borel function $ f $. 
The following lemma proves Lemma \ref{lem: K alpha}. 

\begin{Lem} \label{lem: K alpha2}
Let $ 0 < \alpha < 2 $. Then the variable $ \cR_{\alpha } $ is given by 
\begin{align}
\cR_{\alpha } = 2 \sqrt{ \fe \cT'_{\alpha /2} } 
\label{}
\end{align}
where the variables $ \fe $ and $ \cT'_{\alpha /2} $ are independent. 
\end{Lem}

\begin{proof}[Proof of Lemma \ref{lem: K alpha2}]
Since we have 
\begin{align}
X_{\alpha }(1) \law \sqrt{2} \hat{B}(\cT_{\frac{\alpha }{2}}) , 
\label{}
\end{align}
we obtain the following expression: 
\begin{align}
p^{(\alpha )}_1(x) 
= E \ebra{ \frac{1}{2 \sqrt{ \pi \cT_{\frac{\alpha }{2}} }} 
\exp \kbra{ -\frac{x^2}{4 \cT_{\frac{\alpha }{2}}} } } . 
\label{}
\end{align}
Hence we obtain 
\begin{align}
\frac{p^{(\alpha )}_1(x)}{p^{(\alpha )}_1(0)} 
= E \ebra{ \exp \kbra{ -\frac{x^2}{4 \cT'_{\frac{\alpha }{2}}} } } . 
\label{}
\end{align}

Using an independent exponential variable $ \fe $, we have 
\begin{align}
\frac{p^{(\alpha )}_1(x)}{p^{(\alpha )}_1(0)} 
= E \ebra{ \fe > \frac{x^2}{ 4 \cT'_{\frac{\alpha }{2}} } } 
= E \ebra{ 2 \sqrt{ \fe \cT'_{\frac{\alpha }{2}} } > x } . 
\label{}
\end{align}
Now the proof is complete. 
\end{proof}

\section{Discussions from excursion theoretic viewpoint} 
\label{sec: exc}

Let us recall It\^o's excursion theory (\cite{MR0402949} and \cite{MR0388551}). 
See also the standard textbooks 
\cite{MR1011252} 
and 
\cite{MR1725357}, 
as well as 
\cite{MR2295611}. 

\subsection{It\^o's measure of excursions away from the origin}

We simply write $ \bD $ for the space $ \bD([0,\infty );\bR) $ 
of c\`adl\`ag paths equipped with Skorokhod topology. 
Let $ X=(X(t):t \ge 0) $ be a strong Markov process with paths in $ \bD $ 
starting from 0. 
Suppose that the origin is regular, recurrent and an instantaneous state. 
Then it is well-known (see \cite[Thm.V.3.13]{MR0264757}) 
that there exists a local time at the origin, 
which we denote by $ L=(L(t):t \ge 0) $, 
subject to the normalization: 
\begin{align}
E \ebra{ \int_0^{\infty } \e^{-t} \d L(t) } = 1 . 
\label{eq: def of local time}
\end{align}
This is a choice made in this section; 
but later, we may make another choice, which will be indicated as $ L^{(\alpha )}(t) $, 
$ L(t) $ being always subject to \eqref{eq: def of local time}. 
The local time process $ L=(L(t):t \ge 0) $ 
is continuous and non-decreasing almost surely. 
Thus its right-continuous inverse process 
\begin{align}
\tau(l) = \inf \{ t>0: L(t)>l \} 
\label{}
\end{align}
is strictly-increasing. 
By the strong Markov property of $ X $, 
we see that $ \tau(l) $ is a subordinator. 
We define a random set $ D $ to be the set of discontinuities of $ \tau $: 
\begin{align}
D = \{ l>0 : \tau(l)-\tau(l-)>0 \} . 
\label{}
\end{align}
It is obvious that $ D $ is a countable set. 
Now we define a point function $ \vp $ on $ D $ 
which takes values in $ \bD $ 
as follows: For $ l \in D $, 
\begin{align}
\vp(l)(t) = 
\begin{cases}
X(t+\tau(l-)) 
\ & \text{if} \ 0 \le t < \tau(l)-\tau(l-) 
, \\
0 & \text{otherwise}. 
\end{cases}
\label{}
\end{align}
We call $ \vp = (\vp(l):l \in D) $ 
the {\em excursion point process}. 
Then the fundamental theorem of It\^o's excursion theory is stated as follows. 

\begin{Thm}[It\^o {\cite{MR0402949}}; see also Meyer \cite{MR0388551}] 
\label{thm: fund thm of exc}
The excursion point process $ \vp $ is a Poisson point process, i.e.: 
\\
{\rm (i)} 
$ \vp $ is $ \sigma $-discrete almost surely, i.e., 
for almost every sample path, 
there exists a sequence $ \{ U_n \} $ of disjoint measurable subsets of $ \bD $ 
such that $ \bD = \cup_n U_n $ and 
$ \{ l \in D: \vp(l) \in U_n \} $ is a finite set for all $ n $; 
\\
{\rm (ii)} 
$ \vp $ is renewal, i.e., 
$ \vp(\cdot \wedge s) $ and $ \vp(\cdot + s) $ 
are independent for each $ s>0 $. 
\end{Thm}

For a measurable subset $ U $ of $ \bD $, 
we define a point process $ \vp_U:D_U \to \bD $ as 
\begin{align}
D_U = \{ l \in D: \vp(l) \in U \} 
\qquad \text{and} \qquad 
\vp_U = \vp |_{D_U} . 
\label{}
\end{align}
We call $ \vp_U = (\vp_U(l):l \in D_U) $ 
the restriction of $ \vp $ on $ U $. 
The measure on $ \bD $ defined by 
\begin{align}
\vn(U) = E \ebra{ \sharp ( (0,1] \cap D_U ) } 
\label{eq: def of Ito measure}
\end{align}
is called {\em It\^o's measure} of excursions. 

\begin{Cor}[It\^o {\cite{MR0402949}}] \label{thm: fund thm of exc2}
The following statements hold: 
\\
{\rm (i)} 
Let $ \{ U_n \} $ be a sequence of disjoint measurable subsets of $ \bD $. 
Then the point processes $ \{ \vp_{U_n} \} $ are independent; 
\\
{\rm (ii)} 
Let $ U $ be a measurable subset of $ \bD $ such that $ \vn(U)<\infty $. 
Then $ (0,l] \cap D_U $ is a finite set for all $ l>0 $ a.s. 
Set 
\begin{align}
D_U = \{ 0<\kappa_1<\kappa_2<\cdots \} 
, \qquad 
\vp_U(\kappa_n) = u_n , \ n = 1,2,\ldots . 
\label{}
\end{align}
Then: 
\\ \quad 
{\rm (ii-a)} 
$ \{ \kappa_n-\kappa_{n-1}, u_n : n=1,2,\ldots \} $ are independent 
where $ \kappa_0=0 $; 
\\ \quad 
{\rm (ii-b)} 
For each $ n $, 
$ \kappa_n-\kappa_{n-1} $ is exponentially distributed with mean $ 1/\vn(U) $, 
i.e., $ P(\kappa_n-\kappa_{n-1}>l) = \e^{-l\vn(U)} $ for $ l>0 $; 
\\ \quad 
{\rm (ii-c)} 
For each $ n $, 
$ P(u_n \in \cdot) = \vn(\cdot \cap U) / \vn(U) $; 
\\
{\rm (iii)} 
Let $ F(l,u) $ be a non-negative measurable functional on $ (0,\infty ) \times \bD $. 
Then 
\begin{align}
E \ebra{ \exp \kbra{ - \sum_{l \in D} F(l,\vp(l)) } } 
= \exp \kbra{ - \int \cbra{ 1-\e^{- F(l,u)} } \d l \otimes \vn(\d u) } ; 
\label{}
\end{align}
\\
{\rm (iv)} 
Let $ F(t,u) $ be a non-negative measurable functional on $ (0,\infty ) \times \bD $. 
Then 
\begin{align}
E \ebra{ \sum_{l \in D} F(\tau(l-),\vp(l)) } 
= \int E[ F(\tau(l),u) ] \d l \otimes \vn(\d u) . 
\label{}
\end{align}
\end{Cor}

The proofs of Theorem \ref{thm: fund thm of exc} 
and Corollary \ref{thm: fund thm of exc2} 
are also found in 
\cite{MR1011252} 
and 
\cite{MR1725357}. 

For $ u \in \bD $, define 
\begin{align}
\zeta(u) = \sup \{ t \ge 0: u(t) \neq 0 \} . 
\label{}
\end{align}
For each excursion path $ \vp(l) $, $ l \in D $, 
$ \zeta(\vp(l)) $ is finite and called the {\em lifetime} of the path $ \vp(l) $. 
For a measurable subset $ U $ of $ \bD $, we set 
\begin{align}
\tau_U(l) = \sum_{k \in (0,l] \cap D_U} \zeta(\vp_U(k)) . 
\label{}
\end{align}
Note that $ \tau_{\bD}(l) = \tau(l) $, $ l \ge 0 $. 
By Corollary \ref{thm: fund thm of exc2} (iii), we see that 
the process $ (\tau_U(l):l \ge 0) $ is a subordinator 
with Laplace transform $ E[\e^{-\lambda \tau_U(l)}] = \e^{-l \psi_U(\lambda)} $ 
given by 
\begin{align}
\psi_U(\lambda) = \vn \ebra{ 1-\e^{- \lambda \zeta} ; U } 
, \qquad \lambda>0 . 
\label{}
\end{align}
Since $ \psi(\lambda) := \psi_{\bD}(\lambda) < \infty $, 
we have $ \vn(\zeta>t)<\infty $ for all $ t>0 $; 
in particular, we see that 
the measure $ \vn $ is $ \sigma $-finite.

\subsection{Decomposition of first hitting time before and after last exit time} 
\label{sec: Decomp at LET}

We denote the first hitting time of a closed set $ F $ for $ X $ by 
\begin{align}
T_F(X) = \inf \kbra{ t>0 : X(t) \in F } . 
\label{}
\end{align}
In particular, if $ F=\{ a \} $, the closed set consisting of a single point $ a \in \bR $, 
$ T_F(X) $ is nothing else but 
the first hitting time of {\em point} $ a \in \bR $ for $ X $: 
\begin{align}
T_{\{ a \}}(X) = \inf \kbra{ t>0 : X(t) = a } . 
\label{}
\end{align}
The hitting time $ T_{\{ a \}}(X) $ may be decomposed at the last exit time from 0; 
\begin{align}
T_{\{ a \}}(X) = & G_{\{ a \}}(X) + \Xi_{\{ a \}}(X) 
\label{eq: indep sum of first hitting time}
\end{align}
where $ G_{\{ a \}}(X) $ is the last exit time from 0 before $ T_{\{ a \}}(X) $, 
and where $ \Xi_{\{ a \}}(X) $ is the remainder of time after $ G_{\{ a \}}(X) $, 
i.e., 
\begin{align}
G_{\{ a \}}(X) =& \sup \{ t \le T_{\{ a \}}(X): X(t) = 0 \} 
, \qquad 
\Xi_{\{ a \}}(X) = T_{\{ a \}}(X) - G_{\{ a \}}(X) . 
\label{}
\end{align}
The joint law of the random times $ G_{\{ a \}}(X) $ and $ \Xi_{\{ a \}}(X) $ 
is characterised by the following proposition: 

\begin{Prop} \label{prop: decomp}
Let $ a \neq 0 $. 
Then the random times $ G_{\{ a \}}(X) $ and $ \Xi_{\{ a \}}(X) $ are independent. 
Moreover, the law of $ G_{\{ a \}}(X) $ is of {\ID} type. 
The Laplace transforms of $ G_{\{ a \}}(X) $ and $ \Xi_{\{ a \}}(X) $ are given as 
\begin{align}
E \ebra{ \e^{ - \lambda G_{\{ a \}}(X) } } 
= \kbra{ 1 + 
\frac{ \vn \ebra{ 1-\e^{- \lambda \zeta} ; T_{\{ a \}} > \zeta } }
{ \vn (T_{\{ a \}} < \zeta) } 
}^{-1} 
\label{eq: LT of GaX}
\end{align}
and 
\begin{align}
E \ebra{ \e^{ - \lambda \Xi_{\{ a \}}(X) } } 
= \frac{\vn \ebra{ \e^{- \lambda T_{\{ a \}}} ; T_{\{ a \}} < \zeta }}
{ \vn (T_{\{ a \}} < \zeta) } . 
\label{eq: LT of XiaX}
\end{align}
Consequently, the Laplace transform of $ T_{\{ a \}}(X) $ is given as 
\begin{align}
E \ebra{ \e^{ - \lambda T_{\{ a \}}(X) } } 
= \frac{\vn \ebra{ \e^{- \lambda T_{\{ a \}}} ; T_{\{ a \}} < \zeta }}
{\vn \ebra{ 1- \cbra{ \e^{- \lambda \zeta} \cdot 1_{\{ T_{\{ a \}} > \zeta \}} } } } . 
\label{}
\end{align}
\end{Prop}

\begin{proof}
Set 
\begin{align}
U_a = \kbra{ u \in \bD : T_{\{ a \}}(u) < \infty } . 
\label{}
\end{align}
By Corollary \ref{thm: fund thm of exc2} (i), we see that 
$ \vp_{U_a^c} $ and $ \vp_{U_a} $ are independent. 
We remark that $ \vn(U_a) < \infty $; 
in fact, if we supposed otherwise, 
then there would exist a sequence $ \{ t_n \} $ such that $ t_n \to 0 $ decreasingly 
and that $ X(t_n) =a $, 
which contradicts $ X(0+)=X(0)=0 $. 
Set 
\begin{align}
\kappa_a = \inf \{ l>0 : \vp(l) \in U_a \} . 
\label{}
\end{align}
Then, by Corollary \ref{thm: fund thm of exc2} (ii), we see that 
$ \kappa_a $ and $ \vp(\kappa_a) $ are independent. 
Since $ \kappa_a = \inf D_{U_a} $ and $ \vp(\kappa_a) = \vp_{U_a}(\kappa_a) $, 
they are measurable 
with respect to the $ \sigma $-field generated by $ \vp_{U_a} $. 
Hence we see that 
$ \{ \vp_{U_a^c}, \kappa_a , \vp(\kappa_a) \} $ are independent. 
Note that 
\begin{align}
G_{\{ a \}}(X) = \tau_{U_a^c}(\kappa_a) 
\qquad \text{and} \qquad 
\Xi_{\{ a \}}(X) = T_{\{ a \}}(\vp(\kappa_a)) . 
\label{}
\end{align}
Thus we conclude that $ G_{\{ a \}}(X) $ and $ \Xi_{\{ a \}}(X) $ are independent. 
Moreover, we see that the law of $ G_{\{ a \}}(X) $ is of {\ID} type; 
in fact, $ \tau_{U_a^c} $ is a subordinator with Laplace exponent $ \psi_{U_a^c}(\lambda) $ 
and $ \kappa_a $ is an exponential variable with mean $ 1/\vn(U_a) $ 
independent of $ \tau_{U_a^c} $. 
The law of $ \Xi_{\{ a \}}(X) $ is given by 
\begin{align}
P( \Xi_{\{ a \}}(X) \in \cdot ) = \frac{\vn(u \in \bD: T_{\{ a \}}(u) \in \cdot)}{\vn(U_a)} . 
\label{}
\end{align}
Now the proof is completed. 
\end{proof}

\subsection{Excursion durations}

Let us consider the excursion straddling $ t $. 
For a general study in the setup of linear diffusions, see 
\cite{MR1454113}. 

We define the last exit time from 0 before $ t $ 
and the first hitting time of point 0 after $ t $ 
as follows: 
\begin{align}
G_t(X) = \sup \{ s \le t: X(s)=0 \} 
\qquad \text{and} \qquad 
D_t(X) = \inf \{ s > t : X(s)=0 \} . 
\label{}
\end{align}
We define 
\begin{align}
\Xi_t(X) = t - G_t(X) 
\qquad \text{and} \qquad 
\Delta_t(X) = D_t(X) - G_t(X) . 
\label{}
\end{align}
We recall (see \eqref{eq: def of local time}) that 
$ L=(L(t):t \ge 0) $ denotes the local time at 0 of $ X $ 
and $ \tau=(\tau(l):l \ge 0) $ its right-continuous inverse. 
Then we have 
\begin{align}
G_t(X) = \tau(L(t)-) 
, \qquad 
D_t(X) = \tau(L(t)) 
\label{}
\end{align}
and 
\begin{align}
\Xi_t(X) = t - \tau(L(t)-) 
, \qquad 
\Delta_t(X) = \tau(L(t)) - \tau(L(t)-) . 
\label{}
\end{align}

If the local time process $ L $ 
has the self-similarity property with index $ \gamma $: 
\begin{align}
\cbra{ \frac{L(ct)}{c^{\gamma }}:t \ge 0 } \law (L(t):t \ge 0) 
, \qquad c>0 , 
\label{}
\end{align}
then we have 
\begin{align}
\kbra{ \cbra{ \frac{L(ct)}{c^{\gamma }} :t \ge 0 } , 
\cbra{ \frac{\tau(c^{\gamma }l)}{c}:l \ge 0 } } 
\law 
\kbra{ (L(t):t \ge 0) , (\tau(l):l \ge 0) } 
\label{eq: self-similarity}
\end{align}
for any $ c>0 $; 
in particular, 
$ \tau $ is a stable subordinator of index $ \gamma $. 
Hence the index $ \gamma $ must be in $ (0,1) $. 

We now state two explicit results, 
the proofs of which are postponed after commenting about these results. 

\begin{Thm} \label{thm: age and rem}
Suppose that the local time process 
has the self-similarity property of index $ 0<\gamma<1 $. Then 
\begin{align}
\cbra{ \Xi_1(X) , \Delta_1(X) } 
\law 
\cbra{ \cB_{1-\gamma,\gamma} , \frac{\cB_{1-\gamma,\gamma}}{\cU^{\frac{1}{\gamma}}} } 
\label{eq: age and rem}
\end{align}
where $ \cB_{1-\gamma,\gamma} $ is a beta variable of index $ (1-\gamma,\gamma) $ 
and $ \cU $ is an independent uniform variable on $ (0,1) $. 
\end{Thm}

The following is a special case of Winkel \cite[Cor.1]{MR2144899}: 

\begin{Thm}[\cite{MR2144899}] \label{thm: age and rem at exp time}
Suppose that the local time process 
has the self-similarity property of index $ 0<\gamma<1 $. 
Let $ \fe $ be an independent exponential time. Then 
\begin{align}
\cbra{ G_{\fe}(X) , \Xi_{\fe}(X) , \Delta_{\fe}(X) } 
\law 
\cbra{ \cG_{\gamma} , \hat{\cG}_{1-\gamma} , 
\frac{\hat{\cG}_{1-\gamma} }{\cU^{\frac{1}{\gamma}}} } 
\label{eq: age and rem at exp time}
\end{align}
where 
$ \cG_{\gamma} $ and $ \hat{\cG}_{1-\gamma} $, respectively, are independent gamma variables 
of indeces $ \gamma $ and $ 1-\gamma $, respectively, 
and $ \cU $ is an independent uniform variable. 
\end{Thm}

Generalizing a self-decomposability result of Bondesson (see \cite[Ex.5.6.3]{MR1224674}), 
Bertoin--Fujita--Roynette--Yor \cite[Thm.1.1]{MR2325310} 
and Roynette--Vallois--Yor \cite[Thm.5]{RVY-GGC} 
have recently proved the following: 

\begin{Thm}[\cite{MR2325310} and \cite{RVY-GGC}]
For any $ \gamma \in (0,1) $, 
the laws 
\begin{align}
\frac{\cG_{1-\gamma} }{\cU^{\frac{1}{\gamma}}} 
\qquad \text{and} \qquad 
\cbra{ \frac{1}{\cU^{\frac{1}{\gamma}}}-1 } \cG_{1-\gamma} 
\label{}
\end{align}
are both of {\GGC} type with their Thorin measures having total mass $ 1-\gamma $. 
Here 
$ \cG_{1-\gamma} $ is a gamma variable of index $ 1-\gamma $ 
and $ \cU $ is an independent uniform variable. 
\end{Thm}

\begin{Ex}
For a symmetric stable L\'evy process of index $ \alpha $, 
it is well-known (see Kesten \cite{MR0272059} and Bretagnolle \cite{MR0368175}) that 
the origin is regular for itself if and only if $ 1 < \alpha \le 2 $. 
Let $ X_{\alpha } = (X_{\alpha }(t):t \ge 0) $ be the symmetric stable L\'evy process 
of index $ 1 < \alpha \le 2 $. 
Then its local time process is given as 
\begin{align}
L(t) = \lim_{\eps \to 0+} \frac{C}{2 \eps} \int_0^t 1_{\{ |X_{\alpha }(s)| < \eps \}} \d s 
\label{}
\end{align}
for some constant $ C $. 
Since $ X $ satisfies the self-similarity property with index $ 1/\alpha $, 
so does $ L $ with index $ 1-1/\alpha $, 
and hence Theorems \ref{thm: age and rem} and \ref{thm: age and rem at exp time} 
hold with $ \gamma = 1-1/\alpha $. 
\end{Ex}

\begin{Ex}
For a Bessel process of dimension $ d $, 
it is well-known that the origin is regular for itself 
if and only if $ 0<d<2 $. 
Let $ X=(X(t):t \ge 0) $ be a reflecting Bessel process starting from 0 
of dimension $ d = 2 - 2 \alpha $, $ 0<d<2 $ (or $ 0<\alpha <1 $) 
which is scaled so that 
it has natural scale and speed measure $ m(0,x)=x^{\frac{1}{\alpha }-1} $. 
Then its local time process is given as 
\begin{align}
L(t) = \lim_{\eps \to 0+} \frac{C}{m(0,\eps)} 
\int_0^t 1_{\{ |X(s)| < \eps \}} \d s 
\label{}
\end{align}
for some constant $ C $. 
Since $ X $ satisfies the self-similarity property with index $ \alpha $, 
so does $ L $ with the same index $ \alpha $, 
and hence Theorems \ref{thm: age and rem} and \ref{thm: age and rem at exp time} 
hold with $ \gamma = \alpha $. 
For the relations among several choices in the literature, 
see \cite{MR2417969}. 
\end{Ex}

\begin{Ex}
Let $ \alpha >0 $ and $ 0 < \beta < \min \{ 1,1/\alpha \} $ 
and consider the process $ X=X_{m^{(\alpha )},j^{(\beta)},0,0} $ 
given in \cite[Ex.2.4.(b)]{Y2}. 
Then $ X $ satisfies the self-similarity property with index $ \alpha $, 
but this property seems to have nothing to do with the local time $ L $. 
Since its inverse local time process $ \tau = \eta_{m^{(\alpha )},j^{(\beta)},0,0} $ 
satisfies the self-similarity property with index $ 1/(\alpha \beta) $, 
so does $ L $ with index $ \alpha \beta $, 
and hence Theorems \ref{thm: age and rem} and \ref{thm: age and rem at exp time} 
hold with $ \gamma = \alpha \beta $. 
\end{Ex}

\begin{Rem}
The identities in law 
$ \Xi_1(X_{\alpha }) \law \cB_{\gamma,1-\gamma} $ 
and 
$ D_1(X_{\alpha })-1 \law \cG_{\gamma} / \hat{\cG}_{1-\gamma} $ 
are found in Feller \cite[XIV.3]{MR0270403} 
as the long-time limit laws of similar random variables derived from random walks. 
\end{Rem}

Let us prove Theorems \ref{thm: age and rem} 
and \ref{thm: age and rem at exp time} 
for completeness of this paper. 

\begin{proof}[Proof of Theorem \ref{thm: age and rem}]
Since $ \tau(c^{\gamma }l) \law c \tau(l) $ for $ c,l>0 $, 
we have $ \psi(c \lambda) = c^{\gamma} \psi(\lambda) $. 
Hence we obtain 
\begin{align}
\vn(\zeta \in \d t) = C \frac{\d t}{t^{\gamma + 1}} 
\label{}
\end{align}
for some constant $ C $. 
For $ t>0 $, the excursion straddling time $ t $ is $ \vp(L(t)) $. 
Hence we have 
\begin{align}
G_t = \tau(L(t)-) 
, \qquad 
\Xi_t = t - \tau(L(t)-) 
, \qquad 
\Delta_t = \zeta(\vp(L(t))) . 
\label{}
\end{align}
Let $ p,q,r $ be positive constants. 
Then 
\begin{align}
E \ebra{ \int_0^{\infty } \e^{ - pt - q \Xi_t - r \Delta_t } \d t } 
=& E \ebra{ \sum_{l \in D} \int_{\tau(l-)}^{\tau(l)} \e^{ - pt - q \Xi_t - r \Delta_t } \d t } 
\\
=& E \ebra{ \sum_{l \in D} \e^{-p \tau(l-) - r \zeta(\vp(l))} \int_0^{\zeta(\vp(l))} 
\e^{ - pt - q t } \d t } 
\\
=& \int_0^{\infty } E \ebra{ \e^{-p \tau(l)} } \d l 
\int_{\bD} \vn(\d u) \e^{- r \zeta(u)} \int_0^{\zeta(u)} \e^{ - pt - q t } \d t 
\\
=& \int_0^{\infty } \e^{- l \psi(p)} \d l 
\int_0^{\infty } C \frac{\d s}{s^{\gamma+1}} \e^{-rs} \int_0^s \e^{ -pt-qt } \d t 
\\
=& \frac{C}{\psi(p)} 
\int_0^{\infty } \d t \e^{-pt-qt} \int_t^{\infty } \frac{\d s}{s^{\gamma+1}} \e^{-rs} . 
\label{eq: exc dur 1}
\end{align}
Note that 
\begin{align}
\frac{1}{\psi(p)} 
= \frac{1}{\psi(1) p^{\gamma}} 
= \frac{1}{\psi(1) \Gamma(\gamma)} 
\int_0^{\infty } t^{\gamma-1} \e^{-pt} \d t . 
\label{}
\end{align}
Hence we have 
\begin{align}
\text{\eqref{eq: exc dur 1}} 
=& \frac{C}{\psi(1) \Gamma(\gamma)} \int_0^{\infty } \d t \e^{-pt} 
\int_0^t \d v (t-v)^{\gamma-1} 
\e^{-qv} \int_v^{\infty } \frac{\d s}{s^{\gamma+1}} \e^{-rs} 
\\
=& \frac{C}{\psi(1) \Gamma(\gamma)} \int_0^{\infty } \d t \e^{-pt} 
\int_0^1 \d v v^{-\gamma} (1-v)^{\gamma-1} 
\e^{-qvt} \int_1^{\infty } \frac{\d s}{s^{\gamma+1}} \e^{-rsvt} 
\\
=& C' \int_0^{\infty } \d t \e^{-pt} 
E \ebra{ \exp \kbra{ -q \cB_{1-\gamma,\gamma} t 
- r \frac{\cB_{1-\gamma,\gamma}}{\cU^{\frac{1}{\gamma}}} t } } 
\\
=& C' 
E \ebra{ \frac{1}{p + q \cB_{1-\gamma,\gamma} + r \cB_{1-\gamma,\gamma} \cU^{-1/\gamma}} } 
\label{eq: exc dur 2}
\end{align}
for some constant $ C' $. 

On the other hand, 
by the self-similarity property \eqref{eq: self-similarity}, we have 
$ (\Xi_t,\Delta_t) \law (t \Xi_1,t \Delta_1) $ for fixed $ t>0 $, 
and hence we have 
\begin{align}
E \ebra{ \int_0^{\infty } \e^{ - pt - q \Xi_t - r \Delta_t } \d t } 
= E \ebra{ \frac{1}{p + q \Xi_1 + r \Delta_1} } . 
\label{eq: exc dur 3}
\end{align}
Letting $ q,r \to 0+ $ and comparing \eqref{eq: exc dur 2} and \eqref{eq: exc dur 3}, 
we have $ C'=1 $. 
Therefore we obtain the desired identity in law \eqref{eq: age and rem} 
by the uniqueness property of Stieltjes transform. 
\end{proof}

\begin{proof}[Proof of Theorem \ref{thm: age and rem at exp time}]
Note that 
$ (G_{\fe},\Xi_{\fe},\Delta_{\fe}) \law (\fe G_1,\fe \Xi_1,\fe \Delta_1) $ 
by the self-similarity property \eqref{eq: self-similarity}. 
We also note that 
$ (\fe (1-\cB_{1-\gamma,\gamma}),\fe \cB_{1-\gamma,\gamma}) 
\law (\cG_{\gamma},\hat{\cG}_{1-\gamma}) $ 
by the identity in law \eqref{eq: Za Zb}. 
Therefore we obtain the desired identity in law \eqref{eq: age and rem at exp time} 
as an immediate consequence of Theorem \ref{thm: age and rem}. 
\end{proof}

\section{Harmonic transforms of symmetric stable L\'evy processes}
\label{sec: harm}

We keep the notation 
$ X_{\alpha } = (X_{\alpha }(t):t \ge 0) $ 
for the symmetric stable L\'evy process of index $ \alpha $ 
such that 
\begin{align}
P[\e^{i \theta X_{\alpha }(t) }] = \e^{-t|\theta|^{\alpha }} 
, \qquad \theta \in \bR . 
\label{eq: 4.1}
\end{align}
Note that, with \eqref{eq: 4.1}, we have $ X_2(t) \law \sqrt{2} B(t) $. 
We have 
\begin{align}
P(X_{\alpha }(t) \in \d x) = p^{(\alpha )}_t(x) \d x 
\label{}
\end{align}
where 
\begin{align}
p^{(\alpha )}_t(x) 
= \frac{1}{2\pi} \int_{-\infty }^{\infty } 
\e^{-i x \xi } \e^{-t|\xi|^{\alpha }} \d \xi 
= \frac{1}{\pi} \int_0^{\infty } 
\cos( x \xi ) \e^{-t \xi^{\alpha }} \d \xi . 
\label{eq: p alpha}
\end{align}
We suppose that $ 1 < \alpha \le 2 $. Then the Laplace transform 
\begin{align}
u^{(\alpha )}_q(x) 
= \int_0^{\infty } \e^{-qt} p^{(\alpha )}_t(x) \d t 
= \frac{1}{\pi} \int_0^{\infty } 
\frac{\cos( x \xi ) }{q+\xi^{\alpha }} \d \xi 
\label{}
\end{align}
is finite. Define 
\begin{align}
h^{(\alpha )}_q(x) = u^{(\alpha )}_q(0) - u^{(\alpha )}_q(x) 
, \qquad q>0 , \ x \in \bR 
\label{}
\end{align}
and 
\begin{align}
h^{(\alpha )}(x) 
= \lim_{q \to 0+} h^{(\alpha )}_q(x) 
= \lim_{q \to 0+} \{ u^{(\alpha )}_q(0) - u^{(\alpha )}_q(x) \} 
, \qquad x \in \bR . 
\label{eq: hq limit}
\end{align}

\begin{Lem}[See also {\cite[Sec.4.2]{MR2250510}}] 
\label{lem: h-func}
Suppose that $ 1 < \alpha \le 2 $. Then the following assertions hold: 
\\ \quad 
{\rm (i)} 
$ \displaystyle u^{(\alpha )}_q(0) = u^{(\alpha )}_1(0) q^{\frac{1}{\alpha }-1} $ 
for any $ q>0 $ 
where $ u^{(\alpha )}_1(0) 
= \frac{1}{\alpha \pi} \Gamma(1-\frac{1}{\alpha }) \Gamma(\frac{1}{\alpha }) $; 
\\ \quad 
{\rm (ii)} 
$ \displaystyle 
h^{(\alpha )}(x) = h^{(\alpha )}(1) |x|^{\alpha -1} $ 
for any $ x \in \bR $ where 
$ h^{(\alpha )}(1) = \{ 2 \Gamma(\alpha ) \sin \frac{(\alpha -1)\pi}{2} \}^{-1} $; 
\\ \quad 
{\rm (iii)} 
$ \displaystyle \lim_{q \to 0+} \frac{u^{(\alpha )}_q(x)}{u^{(\alpha )}_q(0)} = 1 $ 
for any $ x \in \bR $. 
\end{Lem}

\begin{proof}
The assertion (i) is obvious by definition. 
It is also obvious that 
\begin{align}
h^{(\alpha )}(x) 
= \frac{1}{\pi} \int_0^{\infty } \frac{1-\cos (x \xi)}{\xi^{\alpha }} \d \xi 
= h^{(\alpha )}(1) |x|^{\alpha -1} 
, \qquad x \in \bR . 
\label{}
\end{align}
For the computation: 
\begin{align}
h^{(\alpha )}(1) 
= \frac{1}{\pi} \int_0^{\infty } \frac{1-\cos \xi}{\xi^{\alpha }} \d \xi 
= \kbra{ 2 \Gamma(\alpha ) \sin \frac{(\alpha -1)\pi}{2} }^{-1} , 
\label{}
\end{align}
see Proposition \ref{prop: const} in the Appendix. 
Hence we obtain (ii). 
We obtain the assertion (iii) noting that 
\begin{align}
\frac{u^{(\alpha )}_q(x)}{u^{(\alpha )}_q(0)} 
= 1 - \frac{h^{(\alpha )}_q(x)}{u^{(\alpha )}_q(0)} 
\ \stackrel{q \to 0+}{\longrightarrow} \ 
1 . 
\label{}
\end{align}
\end{proof}

Let $ (L^{(\alpha )}(t):t \ge 0) $ be the unique local time process such that 
\begin{align}
L^{(\alpha )}(t) = \lim_{\eps \to 0+} \frac{1}{2 \eps} 
\int_0^t 1_{\{ |X_{\alpha }(s)| < \eps \}} \d s 
\qquad \text{a.s.} 
\label{eq: occ density}
\end{align}
Then it is well-known (see \cite[Lemma V.1.3]{MR1406564}) that 
\begin{align}
E \ebra { \int_0^{\infty } \e^{-t} \d L^{(\alpha )}(t) } = u^{(\alpha )}_1(0) . 
\label{}
\end{align}
Let $ \vn^{(\alpha )} $ denote It\^o's measure for the process $ X_{\alpha } $ 
corresponding to this normalisation of the local time $ (L^{(\alpha )}(t):t \ge 0) $. 
Remark that 
\begin{align}
L^{(\alpha )}(t) = u^{(\alpha )}_1(0) L(t) 
\quad (t \ge 0) 
, \qquad 
\vn^{(\alpha )} = \frac{1}{u^{(\alpha )}_1(0)} \vn 
\label{}
\end{align}
where $ (L(t):t \ge 0) $ and $ \vn $, respectively, are as defined 
by \eqref{eq: def of local time} and \eqref{eq: def of Ito measure}, respectively.

\begin{Thm}[\cite{YYY} and \cite{Y}] \label{thm: Y}
Suppose that $ 1<\alpha \le 2 $. Then 
\begin{align}
\vn^{(\alpha )}[h^{(\alpha )}(X(t)) ; \zeta>t] = 1 
, \qquad t>0 . 
\label{}
\end{align}
Consequently, 
there exists a unique probability measure $ P^{h^{(\alpha )}} $ on $ \bD $ such that 
\begin{align}
E^{h^{(\alpha )}}[Z_t] = \vn^{(\alpha )} [Z_t h^{(\alpha )}(X(t)) ; \zeta>t ] 
\label{}
\end{align}
for any $ t>0 $ 
and for any non-negative or bounded $ \cF_t $-measurable functional $ Z_t $. 
\end{Thm}

The proof of Theorem \ref{thm: Y} 
can be found in \cite[Thm.4.7]{YYY}, 
so we omit it. 
See \cite[Thm.1.2]{Y} for the proof of Theorem \ref{thm: Y} 
for a fairly general class 
of one-dimensional symmetric L\'evy processes. 
Several aspects of the law of local time process 
will be discussed in Hayashi--K. Yano \cite{HY}.

\begin{Ex}
In the case where $ \alpha =2 $, we have $ X_2(t) = \sqrt{2} B(t) $, 
and we have the following formulae: 
\begin{align}
p^{(2)}_t(x) =& \frac{1}{2 \sqrt{\pi t}} \e^{ - \frac{x^2}{4t} } 
, \qquad t>0 , \ x \in \bR , 
\label{eq: p(2)} \\
u^{(2)}_q(x) =& \frac{1}{2 \sqrt{q}} \e^{ - \sqrt{q} |x| } 
, \qquad q>0 , \ x \in \bR , 
\label{eq: u(2)} \\
h^{(2)}(x) =& \frac{1}{2} |x| 
, \qquad x \in \bR . 
\label{eq: h(2)}
\end{align}
The process $ (\frac{1}{\sqrt{2}} X(t):t \ge 0) $ under $ P^{h^{(2)}} $ is 
nothing else but 
the symmetrised 3-dimensional Bessel process starting from the origin. 
\end{Ex}

\begin{Thm}[\cite{Y}] \label{thm: Y2}
Let $ q>0 $. Then the following assertions are valid: 
\\ 
\quad {\rm (i)} 
Suppose that $ 1<\alpha \le 2 $. 
Then it holds that 
\begin{align}
\lim_{x \to 0} \frac{h^{(\alpha )}_q(x)}{h^{(\alpha )}(x)} = 1 ; 
\label{eq: Y2 1}
\end{align}
\quad {\rm (iia)} 
Suppose that $ 1<\alpha < 2 $. Let $ a \neq 0 $. 
Then it holds that 
\begin{align}
\lim_{x \to 0} \frac{u^{(\alpha )}_q(a-x)-u^{(\alpha )}_q(a)}{h^{(\alpha )}(x)} = 0 ; 
\label{eq: Y2 2}
\end{align}
\quad {\rm (iib)} 
Suppose that $ \alpha = 2 $. Let $ a \neq 0 $. 
Then it holds that 
\begin{align}
\lim_{x \to \pm 0} \frac{u^{(2)}_q(a-x)-u^{(2)}_q(a)}{h^{(2)}(x)} = 
\pm \e^{-\sqrt{q}|a|} . 
\label{eq: Y2 3}
\end{align}
\end{Thm}

The proof of the claim (i) of Theorem \ref{thm: Y2} 
can be found in \cite[Lem.4.4]{Y}, so we omit it. 
The proof of the claim (iib) of Theorem \ref{thm: Y2} 
is immediate from the formulae \eqref{eq: u(2)} and \eqref{eq: h(2)}, 
so we omit it, too. 
The proof of the claim (iia) of Theorem \ref{thm: Y2} 
is immediate from the following estimate: 

\begin{Lem}[\cite{Y}] \label{lem: uq dif estim}
Suppose that $ 1<\alpha < 2 $. Let $ a,x \in \bR $ with $ 0<2|x|<|a| $. 
Then there exists a constant $ C^{(\alpha )}_q $ such that 
\begin{align}
| u^{(\alpha )}_q(a-x)-u^{(\alpha )}_q(a) | 
\le \frac{C^{(\alpha )}_q}{|a|} . 
\label{eq: uq difference estimate}
\end{align}
\end{Lem}

The proof of Lemma \ref{lem: uq dif estim} 
can be found in \cite[Lem.6.2, (i)]{Y} in a rather general setting, 
but we give it for convenience of the reader. 

\begin{proof}[Proof of Lemma \ref{lem: uq dif estim}]
Integrating by parts, we have 
\begin{align}
u^{(\alpha )}_q(a-x)-u^{(\alpha )}_q(a) 
=& \frac{1}{\pi} \int_0^{\infty } \frac{\cos a\xi - \cos (a-x) \xi}{q+\xi^{\alpha }} \d \xi 
\\
=& \frac{1}{\pi} \int_0^{\infty } \kbra{ \varphi(a \xi) - \varphi((a-x) \xi) } 
\frac{\alpha \xi^{\alpha } \d \xi}{(q+\xi^{\alpha })^2} 
\label{}
\end{align}
where 
\begin{align}
\varphi(x) = \frac{\sin x}{x} 
\quad (x \neq 0) 
, \qquad \varphi(0) = 1 . 
\label{}
\end{align}
Since $ \varphi'(x) = \frac{\cos x}{x} - \frac{\sin x}{x^2} $ ($ x \neq 0 $), 
we have 
\begin{align}
|\varphi(a \xi) - \varphi((a-x) \xi)| 
\le& \abra{ \int_{(a-x) \xi}^{a \xi} |\varphi'(y)| \d y } 
\le \abra{ \int_{(a-x) \xi}^{a \xi} \frac{2}{|y|} \d y } . 
\label{}
\end{align}
We change variables: $ y=u \xi $, then we have 
\begin{align}
|\varphi(a \xi) - \varphi((a-x) \xi)| 
\le \abra{ \int_{a-x}^a \frac{2 }{|u|} \d u } 
\le \frac{4|x|}{|a|} . 
\label{}
\end{align}
Thus we have proved the estimate \eqref{eq: uq difference estimate}. 
\end{proof}

Let us prove the claim {\rm (iia)} of Theorem \ref{thm: Y2}.

\begin{proof}[Proof of the claim {\rm (iia)} of Theorem \ref{thm: Y2}]
Without loss of generality, we may suppose that $ 0<2|x|<|a| $. 
Using the estimate \eqref{eq: uq difference estimate}, we obtain 
\begin{align}
\abra{ \frac{u^{(\alpha )}_q(a-x)-u^{(\alpha )}_q(a)}{h^{(\alpha )}(x)} } 
\le \frac{C^{(\alpha )}_q}{|a|} \cdot \frac{|x|^{2-\alpha }}{h^{(\alpha )}(1)} , 
\label{}
\end{align}
which tends to zero since $ \alpha < 2 $. 
Now the proof is complete. 
\end{proof}

\section{First hitting time of a single point for $ X_{\alpha } $}
\label{sec: hitting}

\subsection{The case of one-dimensional Brownian motion}

Let $ B = (B(t):t \ge 0) $ denote the one-dimensional Brownian motion starting from 0. 
We consider the first hitting time of $ a \in \bR $ for $ B $: 
\begin{align}
T_{\{ a \}}(B) = \inf \{ t>0 : B(t) = a \} . 
\label{}
\end{align}
It is well-known (see, e.g., \cite[Prop.II.3.7]{MR1725357}) 
that the law of the hitting time is of {\SD} type 
where its Laplace transform is given as follows: 
\begin{align}
E \ebra{ \e^{i \theta \hat{B}(T_{\{ a \}}(B))} } 
= 
E \ebra{ \e^{- \frac{1}{2} \theta^2 T_{\{ a \}}(B)} } 
= 
\e^{ - | a \theta | } 
, \qquad \theta \in \bR 
\label{eq: Brownian identity in law 1} 
\end{align}
where $ \hat{B} = (\hat{B}(t):t \ge 0) $ 
stands for an independent copy of $ B $. 
The identity 
\eqref{eq: Brownian identity in law 1} 
can be expressed as 
\begin{align}
\hat{B}(T_{\{ a \}}(B)) \law |a| \cC 
, \qquad 
T_{\{ a \}}(B) \law 2 a^2 \cT_{\frac{1}{2}} . 
\label{eq: Brownian identity in law 1'}
\end{align}

Let $ a>0 $. 
Let us consider the random times $ G_{ \{ a \} }(B) $ and $ \Xi_{ \{ a \} }(B) $. 
The following path decomposition is due to Williams 
(see \cite{MR0258130} 
and \cite{MR0350881}; 
see also Prop.VII.4.8 and Thm.VII.4.9 of \cite{MR1725357}): 

\begin{Thm}[\cite{MR0258130} and \cite{MR0350881}] \label{thm: Williams}
The process $ (B(t):0 \le t \le T_{\{ a \}}(B)) $ is identical in law to 
the process $ (Y(t):0 \le t \le T') $ defined as follows: 
\begin{align}
Y(t) = 
\begin{cases}
B_1(t) 
& \text{for} \ 0 \le t < T_{\{ M \}}(B_1) ; \\
B_2(T-t) 
& \text{for} \ T_{\{ M \}}(B_1) \le t < T ; \\
R(T+t) 
& \text{for} \ T \le t \le T' 
\end{cases}
\label{}
\end{align}
where $ M $, $ B_1 $, $ B_2 $ and $ R $ are independent, 
$ M $ is a uniform variable on $ (0,a) $, 
$ B_1 $ and $ B_2 $ are both identical in law to $ B $, 
$ R $ is a 3-dimensional Bessel process starting at 0, 
and $ T $ and $ T' $ are random times defined as follows: 
\begin{align}
T = T_{\{ M \}}(B_1) + T_{\{ M \}}(B_2) 
, \qquad 
T' = T+T_{\{ a \}}(R) . 
\label{}
\end{align}
\end{Thm}

From this path decomposition, we may compute 
the Laplace transform of $ G_{ \{ a \} }(B) $ as follows: 
\begin{align}
E \ebra{ \e^{ -q G_{ \{ a \} }(B) } } 
= \int_0^a \frac{\d m}{a} \cbra{ E \ebra{ \e^{ -q T_{ \{ m \} }(B) } } }^2 
= \frac{1-\e^{-2 \sqrt{2q} a}}{2 \sqrt{2q} a} 
, \qquad q>0 . 
\label{eq: LT of GaB}
\end{align}
We may also compute the Laplace transform of $ \Xi_{ \{ a \} }(B) $ as follows: 
\begin{align}
E \ebra{ \e^{ -q \Xi_{ \{ a \} }(B) } } 
= E \ebra{ \e^{ -q T_{ \{ a \} }(R) } } 
= \frac{\sqrt{2q} a}{\sinh (\sqrt{2q} a)} 
, \qquad q>0 . 
\label{eq: LT of XiaB}
\end{align}
In other words, we have 
\begin{align}
\Xi_{ \{ a \} }(B) \law T_{ \{ a \} }(R) \law a^2 S_1 . 
\label{}
\end{align}

\begin{Rem}
The laws of first hitting times are known to be of {\SD} type 
also for Bessel processes with drift (see Pitman--Yor \cite{MR620995}) 
and of {\ID} type 
for one-dimensional diffusion processes 
(see Yamazato \cite{MR2279151} and references therein). 
\end{Rem}

\subsection{The law of $ T_{\{ a \}}(X_{\alpha }) $}

Let us consider the first hitting time of point $ a \in \bR $ 
for $ X_{\alpha } $ of index $ 1<\alpha \le 2 $: 
\begin{align}
T_{\{ a \}}(X_{\alpha }) 
=& \inf \{ t>0 : X_{\alpha }(t) = a \} 
. 
\label{}
\end{align}
It is well-known (see, e.g., \cite[Cor.II.5.18]{MR1406564}) that 
\begin{align}
E[ \e^{ - q T_{\{ a \}}(X_{\alpha }) } ] 
= \frac{u^{(\alpha )}_q(a)}{u^{(\alpha )}_q(0)} 
, \qquad q>0. 
\label{eq: LT of FHT}
\end{align}

Let $ \hat{X}_{\alpha } = (\hat{X}_{\alpha }(t):t \ge 0) $ 
be an independent copy of $ X_{\alpha } $. 
The following is a generalization of the formulae 
\eqref{eq: Brownian identity in law 1} 
and \eqref{eq: Brownian identity in law 1'}. 

\begin{Thm} \label{thm: HSalpha}
Suppose that $ 1<\alpha \le 2 $. Let $ a \in \bR $. Then 
\begin{align}
E \ebra{ \e^{ i \theta \hat{X}_{\alpha }(T_{\{ a \}}(X_{\alpha})) } } 
= E \ebra{ \e^{ - |\theta|^{\alpha } T_{\{ a \}}(X_{\alpha}) } } 
= \frac{\sin (\pi/\alpha )}{2 \pi / \alpha } L_{\alpha }(a \theta) 
\label{eq: TaSalpha1}
\end{align}
and 
\begin{align}
\hat{X}_{\alpha }(T_{\{ a \}}(X_{\alpha})) 
\law 
|a| \cC_{\alpha } 
, \qquad 
T_{\{ a \}}(X_{\alpha }) 
\law \frac{|a|^{\alpha }}{(\cR_{\alpha })^{\alpha } \cB_{1-\gamma , \gamma } } 
\label{eq: TaSalpha2}
\end{align}
where $ \gamma = 1/\alpha $. 
\end{Thm}

We can recover \eqref{eq: Brownian identity in law 1'} if we take $ \alpha =2 $, 
noting that 
\begin{align}
T_{\{ a \}}(B) 
\law T_{\{ \sqrt{2} a \}}(X_2) 
\law \frac{2 a^2}{(\cR_2)^2 \cB_{1/2,1/2}} 
\law \frac{a^2}{2 \fe \cB_{1/2,1/2}} 
\law \frac{a^2}{2 \cG_{1/2}} 
\law 2 a^2 \cT_{1/2} . 
\label{}
\end{align}

\begin{proof}[Proof of Theorem \ref{thm: HSalpha}]
If we take $ q=|\theta|^{\alpha } $, then 
\begin{align}
u^{(\alpha )}_q(x) 
= \frac{1}{\pi} \int_0^{\infty } 
\frac{\cos( x \xi ) }{|\theta|^{\alpha }+|\xi|^{\alpha }} \d \xi 
= \frac{|\theta|^{1-\alpha }}{\pi} \int_0^{\infty } 
\frac{\cos( \theta x \xi ) }{1+|\xi|^{\alpha }} \d \xi 
, \qquad x \in \bR . 
\label{}
\end{align}
Hence, by the formula \eqref{eq: LT of FHT}, we obtain 
\begin{align}
E[ \e^{ - |\theta|^{\alpha } T_{\{ a \}}(X_{\alpha }) } ] 
= E[ \cos (\theta |a| \cC_{\alpha }) ] 
= E[ \e^{ i \theta |a| \cC_{\alpha } } ] . 
\label{}
\end{align}
This shows \eqref{eq: TaSalpha1} and the first identity of \eqref{eq: TaSalpha2}. 

To prove the second identity of \eqref{eq: TaSalpha2}, 
it suffices to prove the claim when $ a=1 $; 
in fact, by the self-similarity property, we have 
\begin{align}
T_{\{ a \}}(X_{\alpha }) 
\law |a|^{\alpha } 
T_{\{ 1 \}}(X_{\alpha }) . 
\label{}
\end{align}
Note that 
\begin{align}
\int_0^{\infty } \e^{-qt} P( T_{\{ 1 \}}(X_{\alpha }) < t ) \d t 
= \frac{1}{q} E[ \e^{-q T_{\{ 1 \}}(X_{\alpha })} ] 
= \frac{u^{(\alpha )}_q(1)}{qu^{(\alpha )}_q(0)} . 
\label{}
\end{align}
Since $ u^{(\alpha )}_q(0) = u^{(\alpha )}_1(0) q^{\gamma -1 } $ 
where $ \gamma = \frac{1}{\alpha } $, 
we have 
\begin{align}
\frac{1}{qu^{(\alpha )}_q(0)} 
= \frac{1}{u^{(\alpha )}_1(0)} q^{-\gamma} 
= \frac{1}{u^{(\alpha )}_1(0) \Gamma(\gamma)} \int_0^{\infty } y^{\gamma -1} \e^{-qy} \d y . 
\label{}
\end{align}
Hence, by Laplace inversion, we obtain 
\begin{align}
P( T_{\{ 1 \}}(X_{\alpha }) < t ) 
= \frac{1}{u^{(\alpha )}_1(0) \Gamma(\gamma)} 
\int_0^t (t-s)^{\gamma -1} p^{(\alpha )}_s(1) \d s. 
\label{}
\end{align}
By the scaling property 
\begin{align}
p^{(\alpha )}_s(x) = \frac{1}{s^{\gamma}} p^{(\alpha )}_1 \cbra{ \frac{x}{s^{\gamma}} } 
, \qquad s>0 , 
\label{}
\end{align}
we obtain 
\begin{align}
P( T_{\{ 1 \}}(X_{\alpha }) < t ) 
=& \frac{\Gamma(1-\gamma) p^{(\alpha )}_1(0) }{u^{(\alpha )}_1(0) } 
\int_0^t \frac{(t-s)^{\gamma -1} s^{-\gamma}}{\Gamma(\gamma) \Gamma(1-\gamma)} 
\cdot \frac{p^{(\alpha )}_1 (1/s^{\gamma}) }{p^{(\alpha )}_1(0) } \d s 
\\
=& 
\int_0^1 \frac{(1-s)^{\gamma -1} s^{-\gamma}}{\Gamma(\gamma) \Gamma(1-\gamma)} 
\cdot \frac{p^{(\alpha )}_1 (1/(ts)^{\gamma}) }{p^{(\alpha )}_1(0) } \d s 
\\
& \qquad \text{(from \eqref{eq: p alpha 1 0} and (i) of Lemma \ref{lem: h-func})}
\n \\
=& 
P \cbra{ \cR_{\alpha } > \frac{1}{(t \cB_{1-\gamma , \gamma })^{\gamma}} } 
\qquad \text{(from \eqref{eq: def of cR alpha})}
\\
=& 
P \cbra{ \frac{1}{(\cR_{\alpha })^{\alpha } \cB_{1-\gamma , \gamma } } < t } . 
\label{}
\end{align}
Now the proof is complete. 
\end{proof}

\subsection{Laplace transform formula for first hitting time of two points}

For later use, we prepare several important formulae 
concerning Laplace transforms for first hitting time of two points. 

Let us denote the symmetric $ \alpha $-stable process starting from $ x \in \bR $ 
by $ X^x_{\alpha }(t) = x + X_{\alpha }(t) $. 
Suppose that $ 1<\alpha \le 2 $. 
Recall that the Laplace transform of first hitting time of a single point 
is given by (see \eqref{eq: LT of FHT}) 
\begin{align}
\varphi^q_{x \to a} 
:= E \ebra{ \e^{ - q T_{\{ a \}}(X^x_{\alpha }) } } 
= 
\frac{u^{(\alpha )}_q(x-a)}{u^{(\alpha )}_q(0)} . 
\label{eq: LT of hitting time}
\end{align}
The Laplace transform of first hitting time of two points, 
i.e., $ T_{\{ a \}}(X^x_{\alpha }) \wedge T_{\{ b \}}(X^x_{\alpha }) $, 
is given by the following formula: 

\begin{Prop} \label{prop: hitting time of a and b}
Suppose that $ 1<\alpha \le 2 $. Let $ x,a,b \in \bR $. 
Then 
\begin{align}
\varphi^q_{x \to a,b} 
:= E \ebra{ \e^{ - q T_{\{ a \}}(X^x_{\alpha }) \wedge T_{\{ b \}}(X^x_{\alpha }) } } 
= \frac{u^{(\alpha )}_q(x-a) + u^{(\alpha )}_q(x-b)}{u^{(\alpha )}_q(0) + u^{(\alpha )}_q(a-b)} . 
\label{}
\end{align}
\end{Prop}

\begin{proof}
For any closed set $ F $, we denote 
\begin{align}
T^x_F = T_F(X^x_{\alpha }) = \inf \{ t>0: X^x_{\alpha }(t) \in F \} . 
\label{}
\end{align}
Following \cite[p.49]{MR1406564}, we introduce the {\em capacitary measure} as 
\begin{align}
\mu^q_F(A) = q \int E \ebra{ \e^{-qT^z_F}; X^z_{\alpha }(T^z_F) \in A } \d z 
, \qquad 
A \in \cB(\bR) . 
\label{}
\end{align}
Now we apply Theorem II.2.7 of \cite{MR1406564} and obtain 
\begin{align}
\int \mu^q_F(\d y) \int_A u^{(\alpha )}_q(x-y) \d x 
= \int_A E \ebra{ \e^{-qT^x_F} } \d x 
, \qquad 
A \in \cB(\bR) , 
\label{}
\end{align}
where we have used the fact that the process considered is symmetric. 
This implies that 
\begin{align}
E \ebra{ \e^{-qT^x_F} } 
= \int u^{(\alpha )}_q(x-y) \mu^q_F(\d y) . 
\label{}
\end{align}
By definition of $ \mu^q_F $, we obtain 
\begin{align}
E \ebra{ \e^{-qT^x_F} } 
= q \int E \ebra{ \e^{-qT^z_F} u^{(\alpha )}_q(x-X^z_{\alpha }(T^z_F)) } \d z . 
\label{eq: LT of TxF formula}
\end{align}
Now we let $ F = \{ a,b \} $. 
Then we have $ T^x_F = T^x_{\{ a \}} \wedge T^x_{\{ b \}} $. 
Noting that $ X^z_{\alpha }(T^z_F) = a $ or $ b $ almost surely, we have 
\begin{align}
E \ebra{ \e^{-qT^x_F} } 
= C^q_{a \prec b} u^{(\alpha )}_q(x-a) + C^q_{b \prec a} u^{(\alpha )}_q(x-b) 
\label{eq: LT of TxF}
\end{align}
where 
\begin{align}
C^q_{a \prec b} = 
q \int E \ebra{ \e^{-qT^z_F}; T^z_{\{ a \}} < T^z_{\{ b \}} } \d z . 
\label{}
\end{align}
Since $ T^a_F = T^b_F = 0 $ almost surely, we have 
\begin{align}
1 =& C^q_{a \prec b} u^{(\alpha )}_q(0) + C^q_{b \prec a} u^{(\alpha )}_q(a-b) , 
\\
1 =& C^q_{a \prec b} u^{(\alpha )}_q(b-a) + C^q_{b \prec a} u^{(\alpha )}_q(0) . 
\label{}
\end{align}
Hence we obtain 
\begin{align}
C^q_{a \prec b} = C^q_{b \prec a} = 
\frac{1}{u^{(\alpha )}_q(0) + u^{(\alpha )}_q(a-b)} . 
\label{}
\end{align}
Combining this with \eqref{eq: LT of TxF}, we obtain the desired result. 
\end{proof}

The Laplace transform of first hitting time of point $ a $ before hitting $ b $ 
is given by the following formula: 

\begin{Prop} \label{prop: hitting time of a before b}
Suppose that $ 1<\alpha \le 2 $. Let $ x,a,b \in \bR $ with $ a \neq b $. 
Then 
\begin{align}
\varphi^q_{x \to a \prec b} 
:=& E \ebra{ \e^{ - q T_{\{ a \}}(X^x_{\alpha }) } 
; T_{\{ a \}}(X^x_{\alpha }) < T_{\{ b \}}(X^x_{\alpha }) } 
\\
=& 
\frac{u^{(\alpha )}_q(0) u^{(\alpha )}_q(x-a) - u^{(\alpha )}_q(a-b) u^{(\alpha )}_q(x-b)}
{ \{ u^{(\alpha )}_q(0) \}^2 - \{ u^{(\alpha )}_q(a-b) \}^2 } . 
\label{eq: hitting time of a before b}
\end{align}
\end{Prop}

\begin{proof}
Let us keep the notations in the proof of Proposition \ref{prop: hitting time of a and b}. 
Noting that 
\begin{align}
T^x_{\{ a \}} 
= T^x_{\{ b \}} 
+ T^x_{\{ a \}} \circ \theta_{T^x_{\{ b \}}} 
\qquad \text{on $ \{ T^x_{\{ a \}} > T^x_{\{ b \}} \} $}, 
\label{}
\end{align}
we see, by the strong Markov property, that 
\begin{align}
E \ebra{ \e^{ - q T^x_{\{ a \}} } 
; T^x_{\{ a \}} > T^x_{\{ b \}} } 
= \varphi^q_{x \to b \prec a} \varphi^q_{b \to a} . 
\label{}
\end{align}
Thus we have 
\begin{align}
\varphi^q_{x \to a} 
= \varphi^q_{x \to a \prec b} + \varphi^q_{x \to b \prec a} \varphi^q_{b \to a} . 
\label{eq: varphi x to a}
\end{align}
Combining this with the trivial identity 
\begin{align}
\varphi^q_{x \to a,b} = 
\varphi^q_{x \to a \prec b} + \varphi^q_{x \to b \prec a} , 
\label{}
\end{align}
we obtain 
\begin{align}
\varphi^q_{x \to a \prec b} 
= \frac{\varphi^q_{x \to a} - \varphi^q_{b \to a} \varphi^q_{x \to a,b}}{1-\varphi^q_{b \to a}} . 
\label{}
\end{align}
This proves the desired result. 
\end{proof}

\begin{Rem}
The formula \eqref{eq: hitting time of a before b} can be written as 
\begin{align}
& E \ebra{ \e^{ - q T_{\{ a \}}(X^x_{\alpha }) } 
; T_{\{ a \}}(X^x_{\alpha }) < T_{\{ b \}}(X^x_{\alpha }) } 
\\
=& 
\frac{ u^{(\alpha )}_q(x-b) h^{(\alpha )}_q(a-b) 
+ u^{(\alpha )}_q(0) \{ h^{(\alpha )}_q(x-b) - h^{(\alpha )}_q(x-a) \} }
{ \{ u^{(\alpha )}_q(0) + u^{(\alpha )}_q(a-b) \} h^{(\alpha )}_q(a-b) } . 
\label{}
\end{align}
Letting $ q \to 0+ $, we obtain 
\begin{align}
P \cbra{ T_{\{ a \}}(X^x_{\alpha }) < T_{\{ b \}}(X^x_{\alpha }) } 
= \frac{1}{2} \kbra{ 1 + \frac{ |x-b|^{\alpha -1} - |x-a|^{\alpha -1} }{ |a-b|^{\alpha -1} } }, 
\label{}
\end{align}
which is a special case of Getoor's formula \cite[Thm.6.5]{MR0185663}. 
See also \cite[Thm.6.1]{Y} for its application to It\^o's measure 
for symmetric L\'evy processes. 
\end{Rem}

\begin{Rem}
Let $ a<x<b $. 
Then, as corollaries 
of Propositions \ref{prop: hitting time of a and b} and \ref{prop: hitting time of a before b}, 
we recover the following well-known formulae 
(see, e.g., \cite[Problem 1.7.6]{MR0345224}) 
for the Brownian motion $ (B^x(t) = x + B(t):t \ge 0) $ starting from $ x $: 
\begin{align}
E \ebra{ \e^{ - q T_{\{ a \}}(B^x) \wedge T_{\{ b \}}(B^x) } } 
= \frac{\cosh \cbra{ \sqrt{2q} \cbra{ x-\frac{b+a}{2} } } }
{\cosh \cbra{ \sqrt{2q} \cdot \frac{b-a}{2} } } 
\label{}
\end{align}
and 
\begin{align}
E \ebra{ \e^{ - q T_{\{ a \}}(B^x) } 
; T_{\{ a \}}(B^x) < T_{\{ b \}}(B^x) } 
= 
\frac{\sinh \sqrt{2q} (b-x) }
{ \sinh \sqrt{2q} (b-a) } . 
\label{}
\end{align}
\end{Rem}

\subsection{The Laplace transforms of 
$ G_{\{ a \}}(X_{\alpha }) $ and $ \Xi_{\{ a \}}(X_{\alpha }) $} 
\label{sec: LT of Ga and Xia for Salpha}

The following theorem generalises 
the formulae \eqref{eq: LT of GaB} and \eqref{eq: LT of XiaB}: 

\begin{Thm} \label{thm: LT of Ga and Xia of Salpha}
Suppose that $ 1<\alpha \le 2 $. Let $ a \neq 0 $. Then it holds that 
\begin{align}
E \ebra{ \e^{ - q G_{\{ a \}}(X_{\alpha }) } } 
= \frac{ \{ u^{(\alpha )}_q(0) \}^2 - \{ u^{(\alpha )}_q(a) \}^2}
{2 h^{(\alpha )}(a) u^{(\alpha )}_q(0)} 
\label{eq: LT of Ga Salpha}
\end{align}
and that 
\begin{align}
E \ebra{ \e^{ -q \Xi_{\{ a \}}(X_{\alpha }) } } 
= \frac{u^{(\alpha )}_q(a)}{u^{(\alpha )}_q(0)} 
\cdot \frac{2 h^{(\alpha )}(a) u^{(\alpha )}_q(0)}
{ \{ u^{(\alpha )}_q(0) \}^2 - \{ u^{(\alpha )}_q(a) \}^2} . 
\label{eq: LT of Xia Salpha}
\end{align}
\end{Thm}

\begin{Rem}
The left hand sides of \eqref{eq: LT of Ga Salpha} and \eqref{eq: LT of Xia Salpha} 
are functions of $ q |a|^{\alpha } $ 
since 
\begin{align}
G_{\{ a \}}(X_{\alpha }) \law |a|^{\alpha } G_{\{ 1 \}}(X_{\alpha }) 
\qquad \text{and} \qquad 
\Xi_{\{ a \}}(X_{\alpha }) \law |a|^{\alpha } \Xi_{\{ 1 \}}(X_{\alpha }) . 
\label{}
\end{align}
We may check that so are the right hand sides 
by the following formulae: 
\begin{align}
u^{(\alpha )}_q(0) = |a|^{\alpha -1} u^{(\alpha )}_{q |a|^{\alpha }}(0) 
, \ 
u^{(\alpha )}_q(a) = |a|^{\alpha -1} u^{(\alpha )}_{q |a|^{\alpha }}(1) 
\ \text{and} \ 
h^{(\alpha )}(a) = |a|^{\alpha -1} h^{(\alpha )}(1) . 
\label{}
\end{align}
\end{Rem}

For the proof of Theorem \ref{thm: LT of Ga and Xia of Salpha}, 
we need the following proposition. 

\begin{Prop} \label{prop: vn laplace transform 1}
Suppose that $ X=X_{\alpha } $ with $ 1<\alpha \le 2 $. 
Let $ a \neq 0 $ and $ q,r>0 $. Then 
\begin{align}
\vn^{(\alpha )} \ebra{ \e^{ - q T_{\{ a \}} - r (\zeta-T_{\{ a \}}) } ; T_{\{ a \}} < \zeta } 
= \frac{u^{(\alpha )}_r(a)}{u^{(\alpha )}_r(0)} 
\cdot \frac{u^{(\alpha )}_q(a)}
{ \{ u^{(\alpha )}_q(0) \}^2 - \{ u^{(\alpha )}_q(a) \}^2 } . 
\label{eq: vn laplace transform 1}
\end{align}
Consequently, it holds that 
\begin{align}
\vn^{(\alpha )} \ebra{ \e^{ - q T_{\{ a \}} } ; T_{\{ a \}} < \zeta } 
= 
\frac{u^{(\alpha )}_q(a)}
{ \{ u^{(\alpha )}_q(0) \}^2 - \{ u^{(\alpha )}_q(a) \}^2 } 
\label{eq: vn laplace transform 2}
\end{align}
and that 
\begin{align}
\vn^{(\alpha )} ( T_{\{ a \}} < \zeta ) 
= \frac{ 1 }{2 h^{(\alpha )}(a)} . 
\label{eq: vn laplace transform 3}
\end{align}
\end{Prop}

\begin{proof}
Let us only prove the formula \eqref{eq: vn laplace transform 1}; 
in fact, from this formula we can obtain 
the formulae \eqref{eq: vn laplace transform 2} and \eqref{eq: vn laplace transform 3} 
immediately by the limit \eqref{eq: hq limit} and Lemma \ref{lem: h-func}. 

Let $ \eps>0 $. By the strong Markov property of $ \vn^{(\alpha )} $, we have 
\begin{align}
& \vn^{(\alpha )} \ebra{ \e^{ - q T_{\{ a \}} - r (\zeta-T_{\{ a \}}) } ; \ \eps < T_{\{ a \}} < \zeta } 
\label{eq: Ehalpha-} \\
=& \e^{-q \eps} \vn^{(\alpha )} \ebra{ 
\left. \cbra{ \varphi^q_{x \to a \prec 0} } \right|_{x=X(\eps)} 
\varphi^r_{0 \to a} 
; \ \eps < T_{\{ a \}} \wedge \zeta } 
\\
=& \e^{-q \eps} \varphi^r_{0 \to a} 
E^{h^{(\alpha )}} \ebra{ 
\left. \cbra{ \frac{\varphi^q_{x \to a \prec 0}}{h^{(\alpha )}(x)} } \right|_{x=X(\eps)} 
; \ \eps < T_{\{ a \}} } . 
\label{eq: Ehalpha}
\end{align}
Here we used Theorem \ref{thm: Y}. Note that 
\begin{align}
\begin{split}
& \frac{\varphi^q_{x \to a \prec 0}}{h^{(\alpha )}(x)} 
\cdot \kbra{ \{ u^{(\alpha )}_q(0) \}^2 - \{ u^{(\alpha )}_q(a) \}^2 } 
\\
=& 
u^{(\alpha )}_q(a) \cdot 
\frac{h^{(\alpha )}_q(x)}{h^{(\alpha )}(x) } 
- u^{(\alpha )}_q(0) \cdot 
\frac{u^{(\alpha )}_q(x-a) - u^{(\alpha )}_q(a)}{h^{(\alpha )}(x) } . 
\end{split}
\label{eq: varphiq xtoaprec0/hx}
\end{align}

Suppose that $ 1<\alpha <2 $. Then we see that 
the right hand side of \eqref{eq: varphiq xtoaprec0/hx} 
converges to $ u^{(\alpha )}_q(a) $ as $ x \to 0 $ 
by Theorem \ref{thm: Y2}. 
Letting $ \eps \to 0+ $ in the identity \eqref{eq: Ehalpha-}-\eqref{eq: Ehalpha}, 
we obtain the formula \eqref{eq: vn laplace transform 1} 
by the dominated convergence theorem. 

Suppose that $ \alpha =2 $. 
We may assume without loss of generality that $ a>0 $. 
Then we see that 
the quantity \eqref{eq: Ehalpha} is equal to 
\begin{align}
\frac{1}{2} \e^{-q \eps} \varphi^r_{0 \to a} 
E^{h^{(2)}} \ebra{ 
\left. \cbra{ \frac{\varphi^q_{x \to a \prec 0}}{h^{(2)}(x)} } \right|_{x=X(\eps)} 
; \ \eps < T_{\{ a \}} , \ X(\eps)>0 } ; 
\label{eq: Ehalpha+}
\end{align}
In fact, $ P^{h^{(2)}} $ is nothing else but the law of the symmetrisation of 
the 3-dimensional Bessel process starting from the origin. 
We also see that the right hand side of \eqref{eq: varphiq xtoaprec0/hx} 
converges to $ 2 u^{(\alpha )}_q(a) $ as $ x \to +0 $ 
by Theorem \ref{thm: Y2}. 
Hence, 
letting $ \eps \to 0+ $ in the identity \eqref{eq: Ehalpha-}-\eqref{eq: Ehalpha+}, 
we obtain the formula \eqref{eq: vn laplace transform 1} 
by the dominated convergence theorem. 
Now the proof is complete. 
\end{proof}

Now we proceed to prove Theorem \ref{thm: LT of Ga and Xia of Salpha}. 

\begin{proof}[Proof of Theorem \ref{thm: LT of Ga and Xia of Salpha}.]
Using the formulae \eqref{eq: vn laplace transform 3} 
and \eqref{eq: vn laplace transform 1} (with $ r=q $), 
we have 
\begin{align}
& \vn^{(\alpha )} \ebra{ 1-\e^{-q \zeta} ; \ T_{\{ a \}} > \zeta } 
\\
=& \vn^{(\alpha )} \ebra{ 1-\e^{-q \zeta} } 
- \vn^{(\alpha )}( T_{\{ a \}} < \zeta ) 
+ \vn^{(\alpha )} \ebra{ \e^{-q \zeta} ; \ T_{\{ a \}} < \zeta } 
\\
=& 
\frac{1}{u^{(\alpha )}_q(0)} 
- \frac{1}{2 h^{(\alpha )}(a)} 
+ \frac{u^{(\alpha )}_q(a)}{u^{(\alpha )}_q(0)} 
\cdot \frac{u^{(\alpha )}_q(a)}{ \{ u^{(\alpha )}_q(0) \}^2 - \{ u^{(\alpha )}_q(a) \}^2 } 
. 
\label{}
\end{align}
Hence we obtain 
\begin{align}
\frac{ \vn^{(\alpha )} \ebra{ 1-\e^{-q \zeta} ; \ T_{\{ a \}} > \zeta } }{\vn^{(\alpha )}(T_{\{ a \}} < \zeta)} 
= \frac{2 h^{(\alpha )}(a) u^{(\alpha )}_q(0)}{ \{ u^{(\alpha )}_q(0) \}^2 - \{ u^{(\alpha )}_q(a) \}^2} - 1 . 
\label{}
\end{align}
Combining this with the formula \eqref{eq: LT of GaX}, 
we obtain \eqref{eq: LT of Ga Salpha}. 

Using the formulae \eqref{eq: vn laplace transform 2} 
and \eqref{eq: vn laplace transform 3}, 
we have 
\begin{align}
\frac{\vn^{(\alpha )} \ebra{ \e^{-q T_{\{ a \}}} ; T_{\{ a \}} < \zeta } }{\vn^{(\alpha )}( T_{\{ a \}} < \zeta )} 
= \frac{2 h^{(\alpha )}(a) u^{(\alpha )}_q(a)}{ \{ u^{(\alpha )}_q(0) \}^2 - \{ u^{(\alpha )}_q(a) \}^2} . 
\label{}
\end{align}
Combining this with the formula \eqref{eq: LT of XiaX}, 
we obtain \eqref{eq: LT of Xia Salpha}. 
\end{proof}

\subsection{Overshoots at the first passage time of a level} 

For comparison with the description of the law of a first hitting time, 
we recall the law of the overshoot at the first passage time of a level. 
Let $ X_{\alpha } = (X_{\alpha }(t):t \ge 0) $ denote 
the symmetric stable L\'evy process of index $ 0 < \alpha \le 2 $ 
starting from the origin 
such that 
$ E[\e^{i \lambda X_{\alpha }(t)}] = \e^{-t|\lambda|^{\alpha }} $. 

Let us consider the {\em first passage time of level} $ a>0 $ for $ X_{\alpha } $: 
\begin{align}
T_{[a,\infty )}(X_{\alpha }) = \inf \{ t>0 : X_{\alpha }(t) \ge a \} . 
\label{}
\end{align}
The variable $ X_{\alpha }(T_{[a,\infty )}(X_{\alpha })) - a $ 
is the overshoot at the first hitting time of level $ a $. 

The following theorem is due to Ray \cite{MR0105178}, 
although he does not express his result like this: 

\begin{Thm}[\cite{MR0105178}]
Suppose that $ 0<\alpha \le 2 $. 
Let $ a>0 $. Then 
\begin{align}
X_{\alpha }(T_{[a,\infty )}(X_{\alpha })) -a 
\law a \frac{\cG_{1-\frac{\alpha }{2}}}{\hat{\cG}_{\frac{\alpha }{2}}} 
\label{}
\end{align}
where $ \cG_{1-\frac{\alpha }{2}} $ and $ \hat{\cG}_{\frac{\alpha }{2}} $ 
are independent gamma variables of indices $ 1-\frac{\alpha }{2} $ and $ \frac{\alpha }{2} $, 
respectively. 
\end{Thm}

For its multidimensional analogue, 
see Blumenthal--Getoor--Ray \cite{MR0126885}.

\section{First hitting time of a single point for $ |X_{\alpha }| $}
\label{sec: hitting2}

\subsection{The case of one-dimensional reflecting Brownian motion}

We consider the first hitting time of $ a>0 $ 
for the reflecting Brownian motion $ |B|=(|B|(t)) $: 
\begin{align}
T_{\{ a \}}(|B|) =& \inf \{ t>0 : |B(t)| = a \} 
= T_{\{ a \}}(B) \wedge T_{\{ -a \}}(B) 
. 
\label{}
\end{align}
It is well-known (see, e.g., \cite[Prop.II.3.7]{MR1725357}) 
that the law of the hitting time is of {\SD} type 
where its Laplace transforms is given as follows: 
\begin{align}
E \ebra{ \e^{i \theta \hat{B}(T_{\{ a \}}(|B|))} } 
=& 
E \ebra{ \e^{- \frac{1}{2} \theta^2 T_{\{ a \}}(|B|)} } 
= 
\frac{1}{\cosh (a \theta)} 
, \qquad \theta \in \bR . 
\label{eq: Brownian identity in law 2}
\end{align}
The identity \eqref{eq: Brownian identity in law 2} can be expressed as 
\begin{align}
\hat{B}(T_{\{ a \}}(|B|)) \law a \bC_1 \law 2 a \cM_0 
, \qquad 
T_{\{ a \}}(|B|) \law a^2 C_1 . 
\label{eq: Brownian identity in law 2'}
\end{align}
Noting that 
\begin{align}
\frac{1}{\cosh (a \theta)} 
= \frac{2 \e^{-a |\theta|}}{1+\e^{-2 a |\theta|}} 
= 2 \e^{-a |\theta|} \sum_{n=0}^{\infty } (-1)^n \e^{ -2na |\theta| } , 
\label{}
\end{align}
we have the following expansion: 
\begin{align}
E \ebra{ \e^{- q T_{\{ a \}}(|B|)} } 
= 2 \sum_{n=0}^{\infty } (-1)^n E \ebra{ \e^{- q T_{\{ (2n+1)a \}}(B)} } 
, \qquad q>0 . 
\label{eq: expans for BM}
\end{align}

Let us consider the random times $ G_{\{ a \}}(|B|) $ and $ \Xi_{\{ a \}}(|B|) $. 
By means of random time-change, 
Williams' path decomposition (Theorem \ref{thm: Williams}) 
is also valid for the reflecting Brownian motion $ |B| $ instead of $ B $. 
Hence we may compute the Laplace transforms of these variables as follows: 
\begin{align}
E \ebra{ \e^{ -q G_{\{ a \}}(|B|) } } 
= \int_0^a \frac{\d m}{a} \cbra{ E \ebra{ \e^{- q T_{\{ m \}}(|B|)} } }^2 
= \frac{\tanh (\sqrt{2 q} a)}{\sqrt{2 q} a} 
, \qquad q>0 
\label{}
\end{align}
and 
\begin{align}
E \ebra{ \e^{ -q \Xi_{\{ a \}}(|B|) } } 
= E \ebra{ \e^{ -q T_{\{ a \}}(R) } } 
= \frac{\sqrt{2 q} |a|}{\sinh (\sqrt{2 q} |a|)} 
, \qquad q>0 . 
\label{}
\end{align}
In other words, we have 
\begin{align}
G_{\{ a \}}(|B|) \law a^2 T_1 
, \qquad 
\Xi_{\{ a \}}(|B|) 
\law T_{\{ a \}}(R) 
\law a^2 S_1 . 
\label{}
\end{align}

\subsection{Discussions about the Laplace transform of $ T_{\{ a \}}(|X_{\alpha }|) $}

Let us consider the first hitting time of point $ a>0 $ 
for $ |X_{\alpha }| $ of index $ 1<\alpha \le 2 $: 
\begin{align}
T_{\{ a \}}(|X_{\alpha }|) 
=& \inf \{ t>0 : |X_{\alpha }(t)| = a \} 
= T_{\{ a \}}(X_{\alpha }) \wedge T_{\{ -a \}}(X_{\alpha }) 
. 
\label{}
\end{align}

The following theorem generalises 
the Laplace transform formula \eqref{eq: Brownian identity in law 2} 
and the expansion \eqref{eq: expans for BM}. 

\begin{Thm}
Suppose that $ 1 < \alpha \le 2 $. Let $ a \in \bR $. Then 
\begin{align}
E \ebra{ \e^{ -q T_{\{ a \}}(|X_{\alpha }|) } } 
=& \frac{2 u^{(\alpha )}_q(a)}{u^{(\alpha )}_q(0) + u^{(\alpha )}_q(2a)} 
\label{eq: LT Ta |Salpha|} \\
=& 2 \sum_{n=0}^{\infty } (-1)^n 
E \ebra{ \e^{ -q \kbra{ T_{\{ a \}}(X_{\alpha }) 
+ T_{\{ 2a \}}(X_{\alpha }^{(1)}) + \cdots + T_{\{ 2a \}}(X_{\alpha }^{(n)}) } } } 
\label{eq: LT Ta |Salpha|2}
\end{align}
where $ X_{\alpha }^{(1)},\ldots,X_{\alpha }^{(n)},\ldots $ are independent copies 
of $ X_{\alpha } $. 
\end{Thm}

\begin{proof}
Applying Proposition \ref{prop: hitting time of a and b} 
with $ x=0 $ and $ b=-a $, we obtain the first identity \eqref{eq: LT Ta |Salpha|}. 
Expanding the right hand side, we have 
\begin{align}
E \ebra{ \e^{-q T_{\{ a \}}(|X_{\alpha }|)} } 
= \frac{2 u^{(\alpha )}_q(a)}{u^{(\alpha )}_q(0)} 
\sum_{n=0}^{\infty } (-1)^n \kbra{ \frac{u^{(\alpha )}_q(2a)}{u^{(\alpha )}_q(0)} }^n . 
\label{}
\end{align}
Using the formula \eqref{eq: LT of hitting time}, we may rewrite the identity as 
\begin{align}
E \ebra{ \e^{-q T_{\{ a \}}(|X_{\alpha }|)} } 
= 2 E \ebra{ \e^{-q T_{\{ a \}}(X_{\alpha }) } } 
\sum_{n=0}^{\infty } (-1)^n \kbra{ E \ebra{ \e^{-q T_{\{ 2a \}}(X_{\alpha }) } } }^n . 
\label{}
\end{align}
This is nothing else but the second identity \eqref{eq: LT Ta |Salpha|2}. 
\end{proof}

In the case of Brownian motion $ B=(B(t)) $ on one hand, we have 
\begin{align}
T_{\{ a \}}(B) 
+ T_{\{ 2a \}}(B^{(1)}) + \cdots + T_{\{ 2a \}}(B^{(n)}) 
\law 
T_{\{ (2n+1)a \}}(B) 
\label{}
\end{align}
where $ B^{(1)},\ldots,B^{(n)} $ are independent copies of $ B $. 
In the case of symmetric $ \alpha $-stable process $ X_{\alpha } = (X_{\alpha }(t)) $ 
for $ 1<\alpha <2 $ on the other hand, however, 
the law of the sum 
\begin{align}
T_{\{ a \}}(X_{\alpha }) 
+ T_{\{ 2a \}}(X_{\alpha }^{(1)}) + \cdots + T_{\{ 2a \}}(X_{\alpha }^{(n)}) 
\label{}
\end{align}
differs from that of $ T_{\{ (2n+1)a \}}(X_{\alpha }) $. 
In fact, we have the following theorem. 

\begin{Thm}
Suppose that $ 1 < \alpha \le 2 $. Let $ a \in \bR $. 
Then, for any $ q>0 $ and $ n \ge 1 $, 
\begin{align}
E \ebra{ \e^{ -q \kbra{ T_{\{ a \}}(X_{\alpha }) 
+ T_{\{ 2a \}}(X_{\alpha }^{(1)}) + \cdots + T_{\{ 2a \}}(X_{\alpha }^{(n)}) } } } 
< E \ebra{ \e^{ -q T_{\{ (2n+1)a \}}(X_{\alpha }) } } . 
\label{}
\end{align}
\end{Thm}

\begin{proof}
Set 
\begin{align}
D_n = E \ebra{ \e^{ -q T_{\{ (2n+1)a \}}(X_{\alpha }) } } 
- E \ebra{ \e^{ -q T_{\{ (2n-1)a \}}(X_{\alpha }) } } 
E \ebra{ \e^{ -q T_{\{ 2a \}}(X_{\alpha }) } } . 
\label{}
\end{align}
Then it suffices to prove that $ D_n>0 $ for all $ n \ge 1 $. 

Let us keep the notations in the proof of Proposition \ref{prop: hitting time of a and b}. 
Note that 
\begin{align}
D_n 
= \varphi^q_{0 \to (2n+1)a} 
- \varphi^q_{0 \to (2n-1)a} \varphi^q_{0 \to 2a} . 
\label{eq: Dn 1}
\end{align}
Using the formula \eqref{eq: varphi x to a} and the translation invariance, we have 
\begin{align}
\varphi^q_{0 \to (2n+1)a} 
=& \varphi^q_{0 \to (2n+1)a \prec (2n-1)a} 
+ \varphi^q_{0 \to (2n-1)a \prec (2n+1)a} \varphi^q_{(2n-1)a \to (2n+1)a} 
\\
=& \varphi^q_{0 \to (2n+1)a \prec (2n-1)a} 
+ \varphi^q_{0 \to (2n-1)a \prec (2n+1)a} \varphi^q_{0 \to 2a} . 
\label{}
\end{align}
Using the formula \eqref{eq: varphi x to a}, the translation invariance, and the symmetry, 
we have 
\begin{align}
\varphi^q_{0 \to (2n-1)a} 
= \varphi^q_{0 \to (2n-1)a \prec (2n+1)a} 
+ \varphi^q_{0 \to (2n+1)a \prec (2n-1)a} \varphi^q_{0 \to 2a} . 
\label{}
\end{align}
Hence we obtain 
\begin{align}
D_n 
= \varphi^q_{0 \to (2n+1)a \prec (2n-1)a} \kbra{ 1 - \cbra{\varphi^q_{0 \to 2a}}^2 } , 
\label{eq: Dn 2}
\end{align}
which turns out to be positive 
because both $ \varphi^q_{0 \to (2n+1)a \prec (2n-1)a} $ and 
$ \varphi^q_{0 \to 2a} $ are positive and less than 1. 
Now the proof is complete. 
\end{proof}

\begin{Rem}
The consistency of the two formulae \eqref{eq: Dn 1} and \eqref{eq: Dn 2} 
can be confirmed by the formulae 
\eqref{eq: LT of hitting time} and \eqref{eq: hitting time of a before b} as follows: 
\begin{align}
& \varphi^q_{0 \to (2n+1)a \prec (2n-1)a} \kbra{ 1 - \cbra{\varphi^q_{0 \to 2a}}^2 } 
\\
=& \frac{u^{(\alpha )}_q(0) u^{(\alpha )}_q((2n+1)a) 
- u^{(\alpha )}_q(2a) u^{(\alpha )}_q((2n-1)a)}
{ \{ u^{(\alpha )}_q(0) \}^2 - \{ u^{(\alpha )}_q(2a) \}^2 } 
\cdot \kbra{ 1- \cbra{ \frac{u^{(\alpha )}_q(2a)}{u^{(\alpha )}_q(0)} }^2 } 
\\
=& \frac{u^{(\alpha )}_q((2n+1)a)}{u^{(\alpha )}_q(0)} 
- \frac{u^{(\alpha )}_q((2n-1)a)}{u^{(\alpha )}_q(0)} 
\cdot \frac{u^{(\alpha )}_q(2a)}{u^{(\alpha )}_q(0)} 
\\
=& \varphi^q_{0 \to (2n+1)a} 
- \varphi^q_{0 \to (2n-1)a} \varphi^q_{0 \to 2a} . 
\label{}
\end{align}
\end{Rem}

\subsection{The Laplace transforms of 
$ G_{\{ a \}}(|X_{\alpha }|) $ and $ \Xi_{\{ a \}}(|X_{\alpha }|) $}

Since $ |X_{\alpha }|=(|X_{\alpha }(t)|:t \ge 0) $ is a strong Markov process, 
the arguments of Section \ref{sec: Decomp at LET} 
are valid for $ X=|X_{\alpha }| $. 
Let us compute the Laplace transforms of 
$ G_{\{ a \}}(|X_{\alpha }|) $ and $ \Xi_{\{ a \}}(|X_{\alpha }|) $. 

\begin{Thm} \label{thm: LT of Ga and Xia of |Salpha|}
Suppose that $ 1<\alpha \le 2 $. Let $ a > 0 $. Then it holds that 
\begin{align}
E \ebra{ \e^{ - q G_{\{ a \}}(|X_{\alpha }|) } } 
= \frac{ 2 V^{(\alpha )}_q(a) }{ \{ u^{(\alpha )}_q(0) + u^{(\alpha )}_q(2a) \} 
\{ 4h^{(\alpha )}(a)-h^{(\alpha )}(2a) \} } 
\label{eq: LT of Ga |Salpha|}
\end{align}
and that 
\begin{align}
E \ebra{ \e^{ -q \Xi_{\{ a \}}(|X_{\alpha }|) } } 
= \frac{ u^{(\alpha )}_q(a) \{ 4h^{(\alpha )}(a)-h^{(\alpha )}(2a) \} }{ V^{(\alpha )}_q(a) } 
\label{eq: LT of Xia |Salpha|}
\end{align}
where 
\begin{align}
V^{(\alpha )}_q(a) := 
\{ u^{(\alpha )}_q(0) \}^2 
+ u^{(\alpha )}_q(0) u^{(\alpha )}_q(2a) - 2 \{ u^{(\alpha )}_q(a) \}^2 . 
\label{}
\end{align}
\end{Thm}

For the proof of Theorem \ref{thm: LT of Ga and Xia of |Salpha|}, 
we need a certain Laplace transform formula 
for first hitting time of three points. 
Avoiding unnecessary generality, 
we are satisfied with the following special case: 

\begin{Prop} \label{prop: hitting time of 0,a,-a}
Suppose that $ 1<\alpha \le 2 $. Let $ x,a \in \bR $. 
Then 
\begin{align}
\varphi^q_{x \to 0,a,-a} 
:=& E \ebra{ \e^{ - q T_{\{ 0,a,-a \}}(X^x_{\alpha }) } } 
\\
=& C^q_{0 \prec a,-a} u_q(x) + C^q_{a \prec 0,-a} u_q(x-a) + C^q_{-a \prec 0,a} u_q(x+a) 
\label{eq: hitting time of 0,a,-a 1}
\end{align}
where 
\begin{align}
C^q_{0 \prec a,-a} = \frac{ u^{(\alpha )}_q(0) + u^{(\alpha )}_q(2a) - 2 u^{(\alpha )}_q(a) }
{V^{(\alpha )}_q(a)} 
\label{eq: hitting time of 0,a,-a 2}
\end{align}
and 
\begin{align}
C^q_{a \prec 0,-a} = 
C^q_{-a \prec 0,a} = 
\frac{ u^{(\alpha )}_q(0) - u^{(\alpha )}_q(a) }
{V^{(\alpha )}_q(a)} . 
\label{eq: hitting time of 0,a,-a 3}
\end{align}
\end{Prop}

The proof of 
Proposition \ref{prop: hitting time of 0,a,-a} 
is similar to that of 
Proposition \ref{prop: hitting time of a and b} 
based on the identity \eqref{eq: LT of TxF formula} 
with $ F = \{ 0,a,-a \} $, 
so we omit it. 

\begin{Prop} \label{prop: hitting time of a,-a before 0}
Suppose that $ 1<\alpha \le 2 $. Let $ x,a \in \bR $ with $ a \neq 0 $. 
Then 
\begin{align}
\varphi^q_{x \to a,-a \prec 0} 
:=& E \ebra{ \e^{ - q T_{\{ a,-a \}}(X^x_{\alpha }) } 
; T_{\{ a,-a \}}(X^x_{\alpha }) < T_{\{ 0 \}}(X^x_{\alpha }) } 
\\
=& 
\frac{ \varphi^q_{x \to a,-a} - \varphi^q_{0 \to a,-a} \varphi^q_{x \to 0,a,-a} }
{ 1 - \varphi^q_{0 \to a,-a} } 
\\
=& 
\frac{ 
u^{(\alpha )}_q(0) \kbra{ u^{(\alpha )}_q(x-a) + u^{(\alpha )}_q(x+a) } 
- 2 u^{(\alpha )}_q(a) u^{(\alpha )}_q(x) }
{V^{(\alpha )}_q(a)}. 
\label{eq: hitting time of a before 0,-a}
\end{align}
\end{Prop}

The proof of 
Proposition \ref{prop: hitting time of a,-a before 0} 
is similar to that of 
Proposition \ref{prop: hitting time of a before b}, 
so we omit it.

Let $ \vm^{(\alpha )} $ denote It\^o's measure for $ |X_{\alpha }| $ 
corresponding to the local time satisfying \eqref{eq: occ density}. 
The following proposition is crucial 
to the proof of Theorem \ref{thm: LT of Ga and Xia of |Salpha|}. 

\begin{Prop} \label{prop: vn laplace transform v2}
Suppose that $ 1<\alpha \le 2 $. 
Let $ a > 0 $ and $ q,r>0 $. Then 
\begin{align}
\vm^{(\alpha )} \ebra{ \e^{ - q T_{\{ a \}} - r (\zeta-T_{\{ a \}}) } ; T_{\{ a \}} < \zeta } 
= \frac{ u^{(\alpha )}_r(a) }{ u^{(\alpha )}_r(0) } 
\cdot \frac{ 2 u^{(\alpha )}_q(a) }{ V^{(\alpha )}_q(a) } . 
\label{eq: vn laplace transform v2-1}
\end{align}
Consequently, it holds that 
\begin{align}
\vm^{(\alpha )} \ebra{ \e^{ - q T_{\{ a \}} } ; T_{\{ a \}} < \zeta } 
= \frac{ 2 u^{(\alpha )}_q(a) }{ V^{(\alpha )}_q(a) } 
\label{eq: vn laplace transform v2-2}
\end{align}
and that 
\begin{align}
\vm^{(\alpha )} ( T_{\{ a \}} < \zeta ) 
= \frac{ 2 }{ 4h^{(\alpha )}(a)-h^{(\alpha )}(2a) } . 
\label{eq: vn laplace transform v2-3}
\end{align}
\end{Prop}

\begin{proof}[Proof of Proposition \ref{prop: vn laplace transform v2}]
By definitions of $ \vn^{(\alpha )} $ and $ \vm^{(\alpha )} $, we have 
\begin{align}
\vm^{(\alpha )} \ebra{ \e^{ - q T_{\{ a \}} - r (\zeta-T_{\{ a \}}) } ; T_{\{ a \}} < \zeta } 
= \vn^{(\alpha )} \ebra{ \e^{ - q T_{\{ a,-a \}} - r (\zeta-T_{\{ a,-a \}}) } 
; T_{\{ a,-a \}} < \zeta } . 
\label{eq: vn laplace transform v2 PF1}
\end{align}
Let $ \eps>0 $. 
Then we have 
\begin{align}
& \vn^{(\alpha )} \ebra{ \e^{ - q T_{\{ a,-a \}} - r (\zeta-T_{\{ a,-a \}}) } 
; \eps < T_{\{ a,-a \}} < \zeta } 
\\
=& \e^{-q \eps} 
\vn^{(\alpha )} \ebra{ \left. \cbra{ \varphi^q_{x \to a,-a \prec 0} } \right|_{x=X(\eps)} 
\cdot \varphi^r_{a \to 0} 
; \eps < T_{\{ a,-a \}} \wedge \zeta } 
\\
=& \e^{-q \eps} \varphi^r_{a \to 0} 
E^{h^{(\alpha )}} \ebra{ \left. \cbra{ 
\frac{\varphi^q_{x \to a,-a \prec 0}}{h^{(\alpha )}(x)} } \right|_{x=X(\eps)} 
; \eps < T_{\{ a,-a \}} } . 
\label{eq: Ehalpha2}
\end{align}
Here we utilised Theorem \ref{thm: Y}. 
Noting that, by Theorem \ref{thm: Y2}, we have 
\begin{align}
\lim_{x \to 0} 
\frac{u^{(\alpha )}_q(a-x) + u^{(\alpha )}_q(a+x) - 2 u^{(\alpha )}_q(a) }{h^{(\alpha )}(x)} 
= 0 
\label{}
\end{align}
in whichever case where $ 1<\alpha <2 $ or $ \alpha =2 $. 
Hence, 
we utilise Proposition \ref{prop: hitting time of a,-a before 0} 
and obtain 
\begin{align}
\lim_{x \to 0} \frac{\varphi^q_{x \to a,-a \prec 0}}{h^{(\alpha )}(x)} 
= \frac{ 2 u^{(\alpha )}_q(a) }{ V^{(\alpha )}_q(a) } . 
\label{}
\end{align}

Thus, letting $ \eps \to 0+ $ in the formula \eqref{eq: Ehalpha2}, 
we obtain \eqref{eq: vn laplace transform v2-1} by the dominated convergence theorem. 
By letting $ r \to 0+ $ in the formula \eqref{eq: vn laplace transform v2-1}, 
we obtain \eqref{eq: vn laplace transform v2-2}. 
Noting that 
\begin{align}
V^{(\alpha )}_q(a) = 
2 \kbra{ u^{(\alpha )}_q(0) + u^{(\alpha )}_q(a) } 
h^{(\alpha )}_q(a) 
- u^{(\alpha )}_q(0) h^{(\alpha )}_q(2a) , 
\label{}
\end{align}
we have 
\begin{align}
\lim_{q \to 0+} \frac{ 2 u^{(\alpha )}_q(a) }{ V^{(\alpha )}_q(a) } 
= \frac{2}{4h^{(\alpha )}(a)-h^{(\alpha )}(2a)} . 
\label{}
\end{align}
Hence, by letting $ q \to 0+ $ in the formula \eqref{eq: vn laplace transform v2-2}, 
we obtain \eqref{eq: vn laplace transform v2-3}. 
Now the proof is complete. 
\end{proof}

The proof of Theorem \ref{thm: LT of Ga and Xia of |Salpha|} 
is now completely parallel to 
that of Theorem \ref{thm: LT of Ga and Xia of Salpha}. 
Thus we omit the proof of Theorem \ref{thm: LT of Ga and Xia of |Salpha|}.

\section{Appendix: Computation of the constant $ h^{(\alpha )}(1) $}

\begin{Prop} \label{prop: const}
For $ 1 < \alpha <3 $, it holds that 
\begin{align}
\frac{1}{\pi} \int_0^{\infty } \frac{1-\cos x}{x^{\alpha }} \d x 
= \frac{1}{2 \Gamma (\alpha ) \sin \frac{\pi (\alpha -1)}{2}} . 
\label{eq: int 1-cos x / x alpha }
\end{align}
\end{Prop}

As a check, 
the formula \eqref{eq: int 1-cos x / x alpha } 
in the case when $ \alpha =2 $ 
is equivalent 
via integration by parts 
to the well-known formula: 
\begin{align}
\int_0^{\infty } \frac{\sin x}{x} \d x 
= \frac{\pi}{2} . 
\label{}
\end{align}

\begin{proof}
We start with the identity: 
\begin{align}
\int_0^{\infty } x^{\gamma -1} \e^{-z x} \d x = \Gamma(\gamma) z^{-\gamma} 
\label{}
\end{align}
for $ \gamma>0 $ and $ \Re z>0 $. 
For $ 0<\alpha <1 $, $ \eps>0 $ and $ \lambda \in \bR $, 
we set $ \gamma=1-\alpha $ and $ z = \eps - i \lambda $. 
Then we obtain 
\begin{align}
\int_0^{\infty } \frac{\e^{i \lambda x} \e^{- \eps x} }{x^{\alpha }} \d x 
= \Gamma(1-\alpha ) (\eps - i \lambda)^{\alpha -1} . 
\label{eq: i lambda x - eps x}
\end{align}
Using the identity $ \Gamma(2-\alpha ) = (1-\alpha ) \Gamma(1-\alpha ) $ 
and subtracting \eqref{eq: i lambda x - eps x} for $ \lambda=0 $ 
from that for $ \lambda=\lambda $, we obtain 
\begin{align}
\int_0^{\infty } \frac{(1-\e^{i \lambda x}) \e^{- \eps x} }{x^{\alpha }} \d x 
= \Gamma(2-\alpha ) \cdot \frac{ \eps^{\alpha -1} - (\eps - i \lambda)^{\alpha -1} }{1-\alpha } . 
\label{}
\end{align}
Rewriting the right hand side, we obtain 
\begin{align}
\int_0^{\infty } \frac{(1-\e^{i \lambda x}) \e^{- \eps x} }{x^{\alpha }} \d x 
= \Gamma(2-\alpha ) \int_{\eps}^{\eps-i \lambda} z^{\alpha -2} \d z 
\label{eq: 1 - i lambda x - eps x}
\end{align}
where the integration on the right hand side is taken 
over a segment from $ \{ \eps - i l : l \in \bR \} $. 
Since both sides of \eqref{eq: 1 - i lambda x - eps x} 
are analytic on $ 0 < \Re \alpha < 2 $, 
we see, by analytic continuation, that 
the identity \eqref{eq: 1 - i lambda x - eps x} remains true 
for $ 0 < \alpha < 2 $. 

Let us restrict ourselves to the case when $ 1 < \alpha < 2 $. 
Taking the limit $ \eps \to 0+ $ 
on both sides of the identity \eqref{eq: 1 - i lambda x - eps x}, 
we obtain 
\begin{align}
\int_0^{\infty } \frac{1-\e^{i \lambda x}}{x^{\alpha }} \d x 
=& \Gamma(2-\alpha ) \int_0^{-i \lambda} z^{\alpha -2} \d z 
\\
=& \Gamma(2-\alpha ) \cdot \frac{(-i \lambda)^{\alpha -1}}{\alpha -1} 
\label{}
\end{align}
where the branch of $ f(w)=w^{\alpha -1} $ is chosen so that $ f(1) = 1 $. 
Hence we obtain 
\begin{align}
\int_0^{\infty } \frac{1-\e^{i \lambda x}}{x^{\alpha }} \d x 
=& \Gamma(2-\alpha ) \cdot 
\frac{\lambda^{\alpha -1} }{\alpha -1} \e^{-\frac{\pi(\alpha -1)i}{2}} . 
\label{}
\end{align}
Taking the real parts on both sides, we obtain 
\begin{align}
\int_0^{\infty } \frac{1-\cos \lambda x}{x^{\alpha }} \d x 
=& \Gamma(2-\alpha ) \cdot 
\frac{\lambda^{\alpha -1} }{\alpha -1} \cos \frac{\pi(\alpha -1)}{2} . 
\label{}
\end{align}
Letting $ \lambda = 1 $, we obtain 
\begin{align}
\frac{1}{\pi} \int_0^{\infty } \frac{1-\cos x}{x^{\alpha }} \d x 
= \frac{\Gamma(2-\alpha )}{\pi (\alpha -1)} \cdot 
\cos \frac{\pi(\alpha -1)}{2} . 
\label{eq: formula for 1<alpha<2}
\end{align}
(We may find the formula \eqref{eq: formula for 1<alpha<2} 
also in \cite[pp.88]{MR0322926}.) 
By a simple computation, we have 
\begin{align}
\text{(RHS of \eqref{eq: formula for 1<alpha<2})} 
=& \frac{\Gamma(2-\alpha )}{\pi (\alpha -1)} \cdot 
\frac{\sin \pi (\alpha -1)}{2 \sin \frac{\pi (\alpha -1)}{2}} 
\\
=& 
\frac{1}{(\alpha -1) \Gamma(\alpha -1)} 
\cdot \frac{1}{2 \sin \frac{\pi (\alpha -1)}{2}} 
\\
=& 
\frac{1}{2 \Gamma(\alpha ) \sin \frac{\pi (\alpha -1) }{2}} . 
\label{}
\end{align}
Hence we have proved the identity \eqref{eq: int 1-cos x / x alpha } 
when $ 1<\alpha <2 $. 
By analytic continuation, 
the identity \eqref{eq: int 1-cos x / x alpha } 
is proved to be valid also when $ 2 \le \alpha < 3 $. 
Therefore the proof is complete. 
\end{proof}

\def\polhk#1{\setbox0=\hbox{#1}{\ooalign{\hidewidth
  \lower1.5ex\hbox{`}\hidewidth\crcr\unhbox0}}}
  \def\polhk#1{\setbox0=\hbox{#1}{\ooalign{\hidewidth
  \lower1.5ex\hbox{`}\hidewidth\crcr\unhbox0}}}
  \def\polhk#1{\setbox0=\hbox{#1}{\ooalign{\hidewidth
  \lower1.5ex\hbox{`}\hidewidth\crcr\unhbox0}}}

\end{document}